\documentclass[a4paper,11pt]{article}

\usepackage{graphicx}
\usepackage {amsmath}
\usepackage{amsfonts}
\usepackage{graphicx}
\usepackage {amssymb}
\usepackage {amsmath}
\usepackage{amsthm}
\usepackage{hyperref}
\usepackage{movie15}

\DeclareMathOperator{\R}{\mathbb{R}}

\renewcommand{\phi}{\varphi}

\newtheorem{theorem}{\textbf{Theorem}}
\newtheorem{corollary}{\textbf{Corollary}}

\newtheorem{definition}{\textbf{Definition}}

\newtheorem{example}{\textit{Example}}

\title{An approximation scheme for an Eikonal Equation with discontinuous coefficient \thanks{This 
        work was supported by the European Union under the 7th Framework Programme FP7-PEOPLE-2010-ITN SADCO, Sensitivity Analysis for Deterministic Controller Design.}}

\author{
      Adriano Festa\thanks{Imperial College of London, EEE Department}  \and  Maurizio Falcone \thanks{SAPIENZA -  Universit\`a di Roma, Dipartimento di Matematica}}

\begin{document}

\maketitle

\begin{abstract}
We consider the stationary Hamilton-Jacobi equation
\begin{equation}
b_{ij}(x)u_{x_i}u_{x_j}=\left[f(x)\right]^2      \nonumber
\end{equation} 
 where $b$ can vanish  at some points, the right-hand side $f$ is  strictly positive and  is allowed to be discontinuous. More precisely, we consider special class of discontinuities for which the notion of viscosity solution is well-suited.  We propose a semi--Lagrangian scheme for the numerical approximation of  the viscosity solution in the sense of Ishii and we study its properties. We also prove an {\it a-priori} error estimate for the scheme in $L^1(\Omega)$. The last section contains some applications to control and image processing problems.
\par
 \emph{  Hamilton-Jacobi equation,  discontinuous Hamiltonian, viscosity solutions,  semi-- Lagrangian schemes,  {\em a-priori} error estimates.}
\par 
35F30, 35R05, 65N15
\end{abstract}

\pagestyle{myheadings}
\thispagestyle{plain}
\markboth{A. Festa and M. Falcone}{Approximation of a discontinuous Eikonal Equation}

\section{Introduction}
In this paper we study the following boundary value problem. Let $\Omega\subset \R^N$ be an open bounded domain with a Lipschitz boundary $\partial \Omega$, we consider the Dirichlet problem
\begin{equation}\label{EK}
\left\{
\begin{array}{ll}
b_{ij}(x)u_{x_i}u_{x_j}=\left[f(x)\right]^2 & x\in \Omega\\
u(x)=g(x)& x\in \partial \Omega
\end{array}
\right.
\end{equation} 
where $f$ and $g$ are given functions. We focus to the fact that the function $f$ is Borel measurable, possibly discontinuous.\par 
In the  most classical case, the matrix $(b_{ij})$ is the identity matrix and $f$ is a positive function, so the partial differential equation in (\ref{EK}) reduces to
\begin{equation} \label{EK2} 
|Du(x)|=f(x).
\end{equation}
which is the classical form of an eikonal equation.\par

This equation arises in the study of many problems, e.g.  in geometrical optics, computer vision, control theory and robotic navigation. 
In geometrical optics, to describe the propagation of light the eikonal equation appears in the form
\begin{equation}
   \sum_{i,j=1}^N b_{ij}(x)u_{x_i}u_{x_j}(x)=f(x)
\end{equation}
where $b=\sigma \sigma^t$ and $f$ has the meaning of the refraction index of the media where the light rays are passing. Typically, the refraction law applies across surfaces of discontinuity of $f$.\\
 Another example is offered by a classical problem in computer vision,  the Shape-from-Shading model. In this classical inverse problem we come up with the equation
\begin{equation}
\sqrt{1+|Du(x)|^2}=\frac{1}{I(x)}
\end{equation}
if we assume that the light source is vertical and at infinity (so all the rays are parallel) and the object to reconstruct is the graph of the unknown function $u$. In this particular case the brightness $I(x)\in (0,1]$, i.e.  the intensity of light reflected by the object,  can be discontinuous when the object has edges as we will see in more details at the end of this paper. Using classical tools of convex analysis, both equations above can be rewritten in the form (\ref{EK}).   

Another motivation to deal with discontinuous hamiltonians comes directly from  control theory.  In this framework discontinuous functions can be used to represent targets (for example using  $f$ as a characteristic function) and/or state constraints (using  $f$ as an indicator function) \cite{BFZ10}. Clearly, the well-posedness of (\ref{EK}) in the case of continuous $f$ follows from the theory of viscosity solutions for HJ equations, the interested reader can find  the details  in  \cite{B98} and \cite{BCD97} where there are summarized the well-known results introduced by Crandall, Lions, Ishii and other authors. It is interesting to point out that, when the hamiltonian is discontinuous, the knowledge of  $f$ at every point will not guarantee the well-posedness of the problem even in the framework of viscosity solutions. In fact, for equation (\ref{EK}) it can be easily observed  that, even when $f$ is defined point wise and has appropriate discontinuities,  the  value function for the corresponding control problem will not satisfy the equation in the viscosity sense. In order to define viscosity solutions for this case, we use appropriate semicontinuous envelopes of $f$, following some   ideas  introduced by Ishii in \cite{I85}.\par

The notion of viscosity solution in the case of discontinuous Hamiltonian was proposed by Ishii in \cite{I85} where he presents some existence and regularity results. Other results of well-posedness of Hamilton-Jacobi equations in presence of discontinuous coefficients are been presented by various authors in several works (see \cite{BJ90,F93,B93,DMF00}) and in the specific case, present in many applications, of the eikonal equation \cite{T92,NS95}. \par
Our primary goal is to prove convergence for a semi-Lagrangian scheme which has shown to be rather effective in the approximation of Hamilton--Jacobi equations. The results which have been proved for this type of schemes work for convex and non convex hamiltonians but use the uniform continuity of the hamitonian. Moreover, the typical convergence result is given for the  $L^\infty$ norm which is rather natural when dealing with classical viscosity solutions (see e.g. the result by Crandall and Lions \cite{CL84}, Barles and Souganidis \cite{BS91} and the monograph by Falcone and Ferretti \cite{FF13}).  For classical viscosity solutions, at our knowledge, the only two convergence results in $L^1(\Omega)$ has been proved by Lin and Tadmor \cite{T91, LT01} for a central finite difference scheme and by Bokanowwsky et al. \cite{BFZ10} in dimension one. We have also to mention the level set approach for discontinuous solutions proposed by Tsai et al. \cite{TGO01}. Although classical schemes tailored for the the approximation of regular cases with convex hamiltonians can give reasonable results also for some discontinuous hamiltonians, it would be interesting to have a theoretical results in this situation. Deckelnick and Elliott \cite{DE04} have studied a problem where the solution is still Lipschitz continuous although the hamiltonian is discontinuous. In particular, they have proposed a finite difference scheme for  the approximation of (\ref{EK2}) and their scheme is very similar to a finite difference schemes usually applied for regular hamiltonians.  Their contribution is interesting because they prove  an {\em a-priori} error estimate  in $L^\infty(\Omega)$. \par
Although our work  has been also  inspired by their results, we use different techniques and our analysis is devoted to a scheme of semi--Lagrangian type. The benefits of a semi--Lagrangian scheme with respect to a finite differences scheme are  a better ability to follow the  informations driven by the characteristics, the fact that one can use a larger time-step in evolutive problems still having stability and the fact that SL-schemes do not require a structured grid. This peculiarities give us a faster and more accurate approximation in many cases  as it has been reported in the literature (see e.g.  \cite{FF02, CF07} or appendix A of \cite{BCD97}). It is also  important to note that  we prove an {\em a-priori} error estimate which improves the result in \cite{DE04} because  we consider a more general case (\ref{EK}) where also discontinuous viscosity solutions are accepted. 

This paper is organized as follows.\\
  In Section \ref{sec:teo} we recall some definitions and theoretical  results available for discontinuous hamiltonian. Section \ref{sec:semiapprox} is devoted to the presentation of the scheme and to the proof of some properties which will be used in the proof of convergence. In Section \ref{sec:conv} we prove  convergence and establish an {\em a-priori} error estimate giving the rate of convergence in the $L^1(\Omega)$ norm. Finally, in Section \ref{sec:tests} we present our numerical experiments  dealing with control and image processing problems.
\section{The model problem and previous results} \label{sec:teo}
We present, for readers convenience,  some results of well-posedness mainly taken from a work of Soravia \cite{S06}. We also introduce our assumptions, which are summarized below. \par
The boundary data
\begin{equation}\label{H1}
g:\partial \Omega \rightarrow [0,+\infty[ \hbox{ is continuous},
\end{equation}
the matrix of the coefficients satisfies
\begin{equation}\label{H2}
(b_{ij})=(\sigma_{ik})\cdot(\sigma^t_{kj}) 
\end{equation}
where $i,j=1,\dots, N$ and $k=1,\dots, M$ and $(M\leq N)$.
Then $(b_{ij})$  is a symmetric,  positive semidefinite and possibly degenerate matrix, 
\begin{equation}\label{H3}
\sigma(\cdot)\equiv (\sigma_{ik})_{i=1,...N;\phantom{g} k=1,...M}:\overline{\Omega}\rightarrow \R^{NM} \hbox{ is L-Lipschitz continuous.}
\end{equation}
moreover, the function $f:\R^N\rightarrow[\rho,+\infty[$, $\rho>0$ is Borel measurable and possibly discontinuous. \par
We can give an optimal control interpretation of (\ref{EK}), rewriting the differential operator in the following form
\begin{equation}\label{max}
b_{ij}(x)p_i p_j=\sum_{k=1}^M (p \cdot \sigma_k(x))^2=|p\cdot\sigma(x)|^2,
\end{equation}
where  the columns of the matrix $(\sigma_{ik})_{i,k}$ are the vector fields which will be denoted by  $\sigma_k:\Omega\rightarrow \R^N$, $k=1,...M$. We define $M_\sigma=\max_{i}\Sigma_k \sigma_{i,k}$. In this way the eikonal equation (\ref{EK}) becomes, for $a=(a^1,...a^M)\in \R^M$, the following Bellman equation
\begin{equation}\label{BL}
    \max_{|a|\leq 1}\left\{ -Du(x)\cdot \sum_{k=1}^M a^k \sigma_k(x)\right\}=f(x)
\end{equation}
associated to the symmetric controlled dynamics 
\begin{equation}\label{eq:dyn}
    \dot{y}=\sum_{k=1}^M a^k \sigma_k(y), \quad y(0)=x,
\end{equation}
where the measurable functions $a:[0,+\infty[\rightarrow \{a\in \R^M:|a|\leq 1\}$ are the controls. We will denote in the sequel  by  $y_x(\cdot):= y_x(\cdot,a)$ the solutions of \eqref{eq:dyn}. In this system, optimal trajectories  are the geodesics associated to the metric defined by the matrix $(b_{ij})$, and they are straight lines when $(b_{ij})$ is the  identity matrix. Solution of the equation (\ref{BL}) minimize the following functional
\begin{equation}
   J(x,a(\cdot))=\int_{0}^{\tau_x}f(y(t))dt +g(y(\tau_x))
\end{equation}
where $\tau_x(a(\cdot))=\inf\{t:y_x(t,a)\notin\Omega\}$.\par 
Let us  introduce the concept of {\em discontinuous viscosity solution} for (\ref{EK}) introduced by Ishii in \cite{I85}.\par 
Let $f$ be bounded in $\Omega$  and let
\begin{eqnarray}
 f_*(x)=\lim_{r\rightarrow 0^+ }\inf\{f(y):|y-x|\leq r\}\\
 f^*(x)=\lim_{r\rightarrow 0^+ }\sup\{f(y):|y-x|\leq r\}
 \end{eqnarray} 
$f_*$ and $f^*$ are respectively the  lower semicontinuous  and the upper semicontinuous envelope of $f$.
\begin{definition}
A lower (resp. upper) semicontinuous function $u:\Omega\rightarrow \R\cup \{+\infty\}$ (resp. $u:\Omega \rightarrow \R$) is a viscosity super- (resp. sub-) solution of the equation (\ref{EK}) if for every $\phi\in C^1(\Omega)$, $u(x)<+\infty$, and $x\in {argmin}_{x\in\Omega}(u-\phi)$, (resp. $x\in {argmax}_{x\in\Omega}(u-\phi)$), we have
$$ b_{ij}(x)\phi_{x_i}(x)\phi_{x_j}(x)\geq \left[f_*(x)\right]^2, \quad \hbox{(resp. }b_{ij}(x)\phi_{x_i}(x)\phi_{x_j}(x)\leq \left[f^*(x)\right]^2 \hbox{).}$$
A function $u$ is a discontinuous viscosity solution of  (\ref{EK}) if $u^*$ is a subsolution and $u_*$ is a supersolution. 
\end{definition}
\par 
We remind also that the Dirichlet condition is satisfied in the following weaker sense
\begin{definition}
An upper semicontinuous function $u:\overline{U}\rightarrow \R$, subsolution of  (\ref{EK}), satisfies the Dirichlet type boundary condition in the viscosity sense
if for all $\phi\in C^1$ and $x\in\partial \Omega$, $x\in{argmax}_{x\in\overline{\Omega}}(u-\phi)$ such that $u(x)>g(x)$, then we have
$$b_{ij}(x)\phi_{x_i}\phi_{x_j}\leq \left[f^*(x)\right]^2.$$
Lower semicontinuous functions that satisfy a Dirichlet type boundary condition are defined accordingly.
\end{definition}

In order to see how easily uniqueness can fail without proper assumptions on $f$, now that we accepted that envelopes of function should be used let us consider  the 1D equation
\begin{equation}\label{eqese}
   |u'(x)|=f(x), \quad x\in [-2,2], \quad u(-2)=u(2)=0. 
\end{equation}
 with the choice $f(x)=2 \chi_{\bf Q}$, where $\chi_{\bf Q}$ is the characteristic function of the rationals. Then one easily checks that both $u_1\equiv 0$ and $u_2=2-2|x|$ are viscosity solutions. It is clear, that in general we do not have uniqueness of the discontinuous viscosity solution. We add a key assumption on the coefficient $f$.\par

{\sc Assumption A1.} Let us assume that there exist $\eta>0$ and $K\geq 0$ such that for every $x\in\Omega$ there is a direction $n=n_x\in S^{n-1}$ with 
\begin{equation} \label{nocusp}
f(y+rd)-f(y)\leq Kr
\end{equation}
for every $y\in\Omega$, $d\in S^{n-1}$, $r>0$ with $|y-x|<\eta$, $|d-n|<\eta$ and $y+rd\in\Omega$.

%
%

Under Assumption A1 the following comparison theorem holds. This result, under some more general hypotheses, is presented in \cite{S06}.

\begin{theorem}\label{teo:comp}
Let $\Omega$ be an open domain with Lipschitz boundary. Assume (\ref{H1}), (\ref{H2}), (\ref{H3}) and (\ref{nocusp}). Let $u,v:\overline{\Omega}\rightarrow \R$ be respectively an upper and a lower--semicontinuous function, bounded from below, respectively a subsolution and a supersolution of 
$$b_{i,j}(x)u_{x_i}u_{x_j}=\left[f(x)\right]^2, \quad x\in\Omega$$
Let us assume  that $v$ restricted to $\partial \Omega$ is continuous and that $u$ satisfies the Dirichlet type boundary condition.
Suppose moreover that $u$ or $v$ is Lipschitz continuous. Then $u\leq v$ in $\overline{\Omega}$.
\end{theorem}
\par
From this result, it follows directly that we have uniqueness of a continuous solution.

\begin{corollary}\label{cor1}
Assume (\ref{H1}), (\ref{H2}), (\ref{H3}) and (A1). Let $u:\overline{\Omega}\rightarrow\R$ be a continuous, bounded viscosity solution of the problem (\ref{EK}). Then $u$ is unique 
in the class of discontinuous solutions of the corresponding Dirichlet type problem.
\end{corollary}

\begin{example} [Soravia \cite{S02}] \label{ex2}
This example shows that discontinuous solutions may exists without any contradiction with the previous result. This is due to the fact that  Corollary \ref{cor1} does not cover all possible situations. Let us consider the Dirichlet problem
\begin{equation}
\left\{
\begin{array}{cc}
x^2\left(u_x(x,y)\right)^2+\left(u_y(x,y)\right)^2=\left[f(x,y)\right]^2 & ]-1,1[\times]-1,1[\\
u(\pm 1,y)=u(x,\pm 1)=0 & x,y\in[-1,1]
\end{array}
\right.
\end{equation}
where $f(x, y) = 2$, for $x > 0$, and $f(x, y) = 1$ for $x \leq 0$. In this case we have that 
\begin{equation}
b_{i,j}=\left(\begin{array}{cc} x^2 & 0 \\ 0 & 1 \end{array}\right), \quad \sigma(x)=\left(\begin{array}{cc} x & 0 \\ 0 & 1 \end{array}\right),\nonumber
\end{equation}
therefore the Bellman's equation in this case is
\begin{equation}
   \max_{|a|\leq 1}\left\{-Du(x,y)\cdot a_1 (x,0)^T-Du(x,y)\cdot a_2 (0,1)^T  \right\}=f(x,y).
\end{equation}
It is easy to verify that the piecewise continuous
function,
\begin{equation}
u(x,y)=\left\{
\begin{array}{cc}
2(1-|y|) & x\geq 0, |y|>1+\ln x\\
-2ln(x) & x> 0, |y|\leq 1+\ln x\\
\frac{u(-x,y)}{2} & x< 0.
\end{array}
\right.
\end{equation}
is a viscosity solution of the problem. We know, as indirect implication of Corollary \ref{cor1} that there is no continuous solution. We note that all the class of functions with values in $x=0$ between $1-|y|$ and $2(1-|y|)$ are discontinuous viscosity solutions. However, we have that all discontinuous solutions have $u$ as upper semicontinuous envelope.
\end{example}

As shown in Example \ref{ex2}, in general we do not have existence of a continuous solution and, in general, we do no have a unique solution. But restricting ourselves to a special class of solutions, essentially the case presented in the previous example, we can preserve the accuracy of numerical approximations and we can also get an error estimate, as we will see in the sequel. \par
The presence of discontinuities is due to the degeneracy of the coefficient $\sigma$. To handle this case we need some additional hypotheses. In this case, however, the assumption will be given on the interface of degeneracy of $\sigma$.\par
From here we will restrict ourselves to the case $N=2$.\par

Let us denote by $\ell (C)$ the length of a curve $C$ and  assume the existence of a regular curve $\Sigma_0$ which splits the domain $\Omega$ in two non degenerating parts. Calling $\eta(x)=(\eta^1(x), \eta^2(x))$ the unit normal to $\Sigma_0$ on the point $x\in\Sigma_0$, we state:
 
{\sc Assumption A2.} There exists a curve $\Sigma_0\subset \Omega$ such that, for the points $x\in\Sigma_0$ we have
$$ \eta^1(x)\sigma_1(x)+\eta^2(x)\sigma_2(x)=0; $$
moreover
\begin{enumerate}
\item $ p^1(x)\sigma_1(x)+p^2(x)\sigma_2(x) \neq 0$ for every $(p^1,p^2)\in B(0,1)$ and $x\notin {\Sigma_0}$; 
\item $\ell({\Sigma_0})<+\infty$.\par
\item Let $\Omega=\Omega_1 \cup \Omega_2\cup {\Sigma_0} $, where, in each subset $\Omega_j$ there is not degeneracy of $\sigma$, we have $\Omega_j\cap\partial \Omega \neq \emptyset$ for $j=\{1,2\}$.
\end{enumerate}
\par
We conclude this section with the following result, which  can be derived by adapting the classical proof by Ishii \cite{I87}:
\begin{theorem} \label{lipreg}
Let $\Omega$ be an open domain with Lipschitz boundary. Assume (\ref{H1}), (\ref{H2}), (\ref{H3}), \ref{nocusp} and assumptions $A2$. Let $u:\overline{\Omega}\rightarrow\R$ be a  bounded viscosity solution of the problem (\ref{EK}). It is Lipschitz continuous in every set $\Omega_1$ and $\Omega_2$.

\end{theorem}

\begin{proof}
Take a parameter $\delta>0$, and define the set 
\begin{equation}
\Sigma_{\delta}:=\left\{x\in\Omega | B(x,\delta)\cap \Sigma_0 \neq \emptyset \right\};
\end{equation}
we want to study the regularity of the viscosity solution in the set $\overline{\Omega_1\setminus \Sigma_{\delta}}=\overline{\Omega}_1^{\delta}$. \par
In order to describe our boundary assumptions on $\overline{\Omega}_1^{\delta}\cap \Sigma_{\delta}$ let us define $L:\overline{\Omega}_1^{\delta}\times \overline{\Omega}_1^\delta\rightarrow \R$ by
\begin{eqnarray}
L(x,y):=\inf\big\{\int_0^1 N(f^*(\gamma(t)),\gamma'(t))dt | \gamma\in W^{1,\infty}((0,1),\overline{\Omega}_1^{\delta}) \\
\hbox{ with } \gamma(0)=x, \hbox{ } \gamma(1)=y \big\}\nonumber
\end{eqnarray}
where 
\begin{equation}
N(r,\zeta):=\sup \left\{-(\zeta,p)| \max_{|a|\leq 1}\left\{-p\cdot \sum_{k=1}^M a^k \sigma_k (x)=r\right\} \right\}.
\end{equation}
Then we extend the boundary condition to $\overline{\Omega}_1^\delta \cap \Sigma_{\delta}$ in the following way:
\begin{equation}\label{exd}
g(x)=\inf_{y\in \partial \Omega_1^{\delta} \setminus \partial \Sigma_{\delta}} \left\{ g(y)+L(x,y)\right\} \quad x\in \overline{\Omega}_1^{\delta} \cap \Sigma_{\delta}
\end{equation}
We can claim now, that there exists a viscosity solution $u_1^\delta\in C^{0,1}(\overline{\Omega}_1^\delta)$ of \eqref{EK} with the Dirichlet conditions introduced above. This is proved in  Ishii \cite{I87}.\par
We do the same on the set $\Omega_2$, getting the function $u_2^{\delta}\in C^{0,1}(\overline{\Omega}_2^\delta)$. Now the class of functions 
\begin{equation}
u^\delta (x):=
\left\{
\begin{array}{ll}
u_1^\delta & x\in \overline{\Omega}_1^\delta\\
u_2^\delta & x\in \overline{\Omega}_2^\delta
\end{array}
\right.
\end{equation}
in a viscosity solution of \eqref{EK} in $ \overline{\Omega}_1^\delta\cup  \overline{\Omega}_2^\delta$. For the arbitrariness of $\delta$ and defining $u$ on the discontinuity as said previously we get the thesis.
\end{proof}

Which value the solution can assume in $\Sigma_0$? As shown in Example 2 and in accordance with the definition of discontinuous viscosity solutions, we can choose for $x\in\Sigma_0$ every value between $u_* $ and $u^*$. \par

We can observe that in this class we can also include a easier case. If instead the function $\sigma$ we consider a coefficient $c(x):\Omega\rightarrow \R$ where $c(x)\geq 0$ for all $x\in \Omega$ but which can vanish in some points. In particular, in this case we will define $\Sigma_0:=\{x\in\Omega | c(x)=0 \}$ and the previous hypothesis on the nature of $\Sigma_0$ reduces to 
$$ \ell(\Sigma_0)<+\infty \hbox{ and }   \Omega_j\cap\partial \Omega \neq \emptyset \hbox{  for  } j=\{1,2\}.$$

\section{The semi-Lagrangian approximation scheme and its properties}\label{sec:semiapprox}
We construct a semi-Lagrangian approximation scheme for the equation (\ref{EK}) following the approach \cite{FF02} .\par
Introducing the  Kruzkov's change of variable,  $v(x)=1-e^{-u(x)}$ and  using (\ref{BL}) and (\ref{max}) the problem (\ref{EK}) becomes
\begin{equation}
\left\{
\begin{array}{ll}
|Dv(x)\cdot \sigma(x)|=f(x)(1-v(x)) & x\in \Omega\\
v(x)=1-e^{-g(x)}& x\in \partial \Omega
\end{array}
\right.
\end{equation} 
to come back to the original unknown $u$  we can use the inverse transform, i.e. $u(x)=\ln(1-v(x))$.\par
Let us to observe that since  $u(x)\geq 0$, we have $0\leq v(x)<1$.  
We can write the previous equation in the equivalent way
\begin{equation}\label{EIK2}
\left\{
\begin{array}{ll}
v(x)+\frac{1}{f(x)}|D v(x)\cdot \sigma(x)|=1 & x\in \Omega\\
v(x)=1-e^{-g(x)}& x\in \partial \Omega
\end{array}
\right.
\end{equation} 
We want to build a discrete approximation of (\ref{EIK2}). We pass to the Bellman's equation in the following form
\begin{equation}\label{2:pas}
\sup_{a\in B(0,1)} \left\{\frac{\sum_k a^k \sigma_k(x)}{f(x)}\cdot Dv(x)\right\}=1-v(x).
\end{equation}
We observe that, in this formulation, it exists a clear  interpretation of this equation as the value function of an optimization problem of constant running cost and discount factor  equal to one and the modulus of the velocity in the direction $a$ of the dynamics equal to $\frac{a\cdot \sigma(x)}{f(x)}$.\par
We discretize the left-hand side term of (\ref{2:pas}) as a directional derivative and we arrive to the following discrete problem:
\begin{equation}\label{SDE}
\left\{
\begin{array}{ll}
v_h(x)=\frac{1}{1+h}\inf\limits_{a\in B(0,1)}\left\{v_h\left(x-\frac{h}{f(x)}\sum_k a^k \sigma_k(x)\right)\right\}+\frac{h}{1+h} & x\in \Omega\\
v_h(x)=1-e^{-g(x)}& x\in \partial \Omega
\end{array}
\right.
\end{equation} 
where $h$ is a positive real number and we will assume (to simplify the presentation) that ${x-\frac{h}{f(x)}\sum_k a^k \sigma_k(x)\in\overline{\Omega}}$ for every $a\in B(0,1)$. \par
We have to remark that for $x\in\overline{\Omega}$ and a direction $d\in \partial B(0,1)$, we always can find an $a\in B(0,1)$ such that $\frac{a}{|a|}=d$ and $x-\frac{h}{f(x)}\sum_k a^k \sigma_k(x)\in\overline{\Omega}$, (see Figure \ref{2:dom}) because $\Omega$ is an open set and we can chose the variable $a$ null to remain on $x$. \par

\begin{figure}[htb]
\begin{center}
\includegraphics[height=6cm]{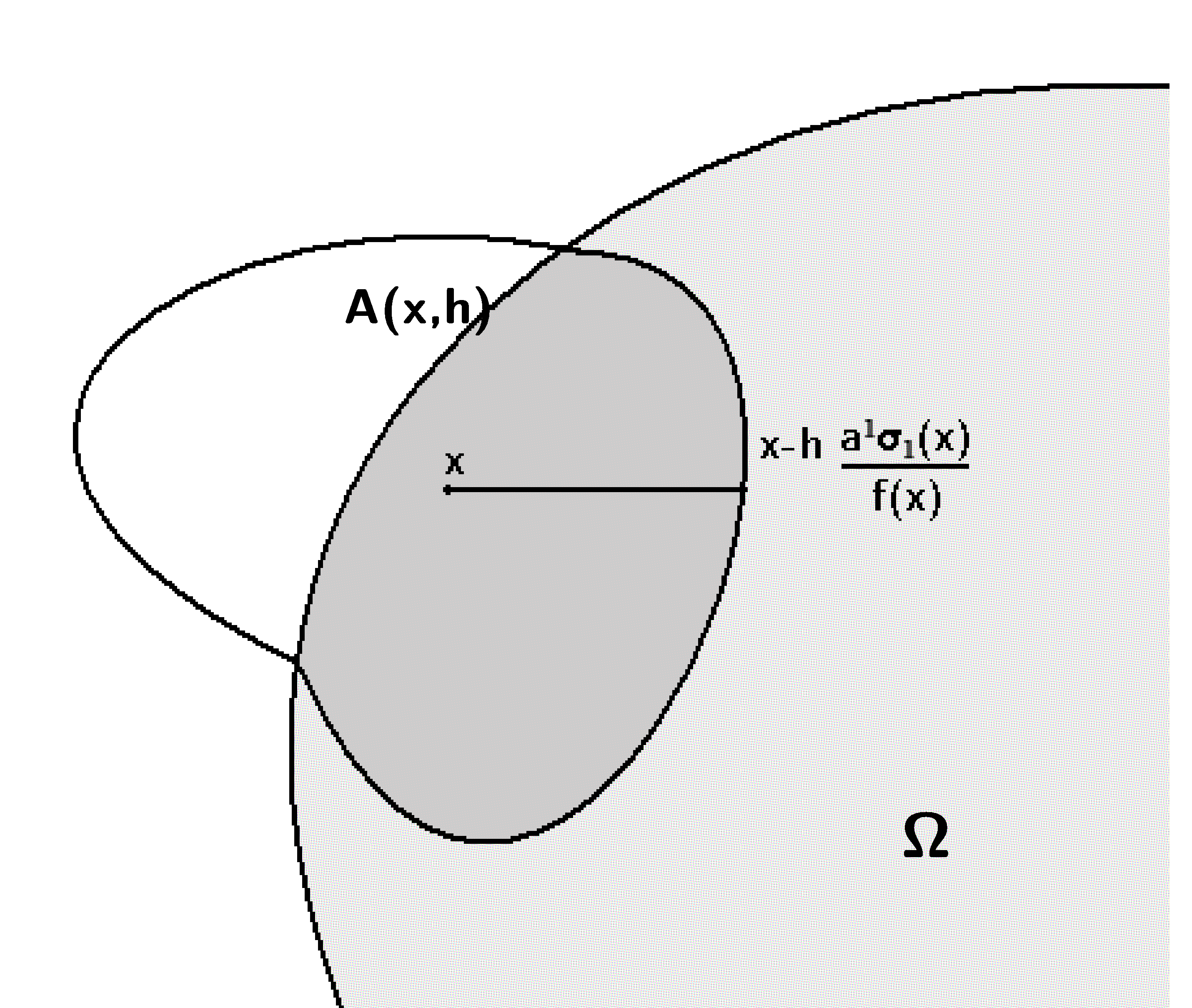}  \\
\caption{The set $A(x,h):=\left\{x-\frac{h}{f(x)}\sum_k a^k \sigma_k(x); a\in B(0,1)\right\}$ in dimension 2. In dark grey $\Omega\cap A(x,h)$} \label{2:dom}
\end{center}
\end{figure}

Let introduce a space discretization of (\ref{SDE}) yielding a fully discrete scheme. We construct a regular triangulation of $\Omega$ made by a family of simplices $S_j$, such that $\overline{\Omega}=\cup_j S_j$, denoting $x_m$, $m=1, ... , L$, the nodes of the triangulation, by 
\begin{equation} \Delta x :=\max_j {\it diam}(S_j)
\end{equation}
the size of the mesh ($diam(B)$ denotes the diameter of the set $B$) and by $G$ the set of the knots of the grid.

We look for a solution of
\begin{equation}\label{DE}
\left\{
\begin{array}{l}
W (x_m)=\frac{1}{1+ {h}}\min\limits_{a \in B(0,1)} I[W](x_m - \frac{h}{f(x_m)}\sum\limits_k a^k \sigma_k(x_m))+\frac{h}{1+h}\quad x_m\in G\\
W(x_m)=1-e^{-g(x_m)}\quad \phantom{fffffffffffffffhhhhhhhhhff}x_m\in G\cap \partial \Omega
\end{array}
\right.
\end{equation} 
where $I[W](x)$ is a linear interpolation of $W$ on the point $x$, in the space of piecewise linear functions on $\overline{\Omega}$
\begin{equation}
\mathcal{W}^{\Delta x}:=\left\{w:\overline{\Omega}\rightarrow \R|w\in C(\Omega) \hbox{ and }Dw(x)=c_j \hbox{ for any } x\in S_j\right\}.\nonumber
\end{equation}

\begin{theorem}
Let $x_m - \frac{h}{f(x_m)}\sum_k a^k \sigma_k(x_m)\in \overline{\Omega}$, for every $x_m \in G $, for any $a\in B(0,1)$, so there exists a unique solution $W$ of (\ref{DE}) in $\mathcal{W}^{\Delta x}$
\end{theorem}
\begin{proof}
By our assumption, starting from any $x_m \in G$ we will reach points which still belong to $\Omega$. So, for every $w\in\mathcal{W}^{\Delta x}$ we have 
\begin{equation}
  w\left(x_m - \frac{h}{f(x_m)}\sum_k a^k \sigma_k(x_m)\right)=\sum_{j=1}^{L}\lambda_{m j}(a) w(x_j) \nonumber
\end{equation}
where $\lambda_{m j}(a)$ are the coefficients of the convex combination representing the point $x_m - \frac{h}{f(x_m)}\sum_k a^k \sigma_k(x_m)$, and $L$ the number of nodes of $G$, i.e.
\begin{equation} \label{ptcrc}
x_m - \frac{h}{f(x_m)}\sum_k a^k \sigma_k(x_m) =\sum_{j=1}^L \lambda_{m j}(a) x_j
\end{equation}
now we observe
\begin{equation} \label{ptcrc2}
0\leq \lambda_{m j}(a)\leq 1 \quad \hbox{ and } \sum_{j=1}^L \lambda_{m j}(a)=1 \quad \hbox{ for any } a\in B(0,1)
\end{equation}
Then (\ref{DE}) is equivalent to the following fixed point problem in finite dimension
\begin{equation}
W=T(W)  \nonumber
\end{equation}
where the map $T:\R^L\rightarrow \R^L$ is defined componentwise as
\begin{equation}
(T(W))_m:=\left[\frac{1}{1+h}\min_{a\in B(0,1)}\Lambda(a)W+\frac{h}{1+h}\right]_m\quad m\in 1,...,L
\end{equation}
$W_m\equiv W(x_m)$ and $\Lambda(a)$ is the $L\times L$ matrix of the coefficients $\lambda_{m j}$ satisfying (\ref{ptcrc}), (\ref{ptcrc2}) for $m,j\in1,...,L$.\par
$T$ is a contraction mapping. In fact, let $\overline{a}$ be a control giving the minimum in $T(V)_m$, we have 
\begin{multline}
\left[T(W)-T(V)\right]_m\leq \frac{1}{1+h}\left[\Lambda(\overline{a})(W-V)\right]_m\\
\leq\frac{1}{1+h}\max_{m,j}|\lambda_{m j}(a)| ||W-V||_\infty\leq\frac{1}{1+h} ||W-V||_\infty
\end{multline}
Switching the role of $W$ and $V$ we can conclude that 
\begin{equation}
\left\|T(W)-T(V)\right\|_{\infty} \leq \frac{1}{1+h}\left\|W-V\right\|_\infty
\end{equation}
\end{proof}

\subsection{Properties of the scheme}
The solution of (\ref{DE}) has the following crucial proprieties:\par
{\em Consistency}\\
From (\ref{DE}) we obtain
\begin{equation}
  W(x_m)-\frac{1}{h}\min_{a\in B(0,1)}\left\{-W(x_m)+I[W](x_m - \frac{h}{f(x_m)}\sum_k a^k \sigma_k(x_m))\right\}=1
\end{equation}
We can see the term on the minimum as a first order approximation of the directional derivative
\begin{equation}
   -\min_{a\in B(0,1)}\left\{DW\cdot \sum_k a^k \sigma_k(x)\right\}+o(h)=1-W(x_m)
\end{equation}
using $\max(\cdot)=-\min(-\cdot)$ we find the consistency, that is of order $o(h+\Delta x)$.\par
{\em Convergence and monotonicity.} 
Since $T$ is a contraction mapping in $\R^N$, the sequence 
\begin{equation}
W^n=T(W^{n-1}),
\end{equation}
will converge to $W$, for any $Z\in \R^N$. Moreover, the following estimate holds true:
\begin{equation}
 ||W^n-W||_\infty\leq \left(\frac{1}{1+h }\right)^n ||W_0-W||_\infty.
\end{equation}

\section{An a-priori estimate in $L^1(\Omega)$} \label{sec:conv}
In this section we present our main  result. As stated in previous section the relevance of this result is due to its applicability in the case of discontinuous value functions. Using a $L^1(\Omega)$ norm we can extend the convergence result also to this class of solutions.\par

\begin{theorem}\label{rate1}
Let assume the hypotheses (\ref{H1}), (\ref{H2}), (\ref{H3}), (\ref{nocusp}) and assumptions on the set $\Sigma_0$. Moreover, let $\frac{h}{\Delta x}<\frac{\rho}{M_\sigma}$.\par 
We have that 
\begin{equation}
 ||v(x)-W(x)||_{L^1(\Omega)}\leq C \sqrt{h}+C'\Delta x \quad \hbox{ for all } h>0
\end{equation}
for some positive constant $C,C'$ independent from $h$ and $\Delta x$.
\end{theorem}

\begin{proof}
We start introducing the set $\Sigma_{\Delta x}$ defined as it follows
$${\Sigma_{\Delta x}}:=\left\{x\in\Omega| B\left(x,\Delta x\right)\cap {\Sigma_0} \neq \emptyset\right\}.$$

We observe that
\begin{multline}\label{pass1}
   ||v(x)-W(x)||_{L^1(\Omega)}\leq \int_{\Omega\setminus {\Sigma_{\Delta x}}} |v(x)-W(x)|dx +\int_{{\Sigma_{\Delta x}}}|v(x)-W(x)|dx\\
   \leq \sum_j\int_{\Omega_j} |v(x)-W(x)|dx +\int_{{\Sigma_{\Delta x}}}|v(x)-W(x)|dx
\end{multline}
where $\Omega:=\cap_j \Omega_j$ is the partition of $\Omega$ generated from ${\Sigma_0}$ as stated in the definition of the set $\Sigma_0$.\par

From the Kruzkov's transform we know  that  $|v(x)-W(x)|\leq 1$ for all $x\in\Omega$ and adding the assumptions on the set $\Sigma_0$ we get, for a fixed $ C'>0$,

\begin{equation}
   \int_{{\Sigma_{\Delta x}} }|v(x)-W(x)|dx \leq \int_{{\Sigma_{\Delta x}}}dx \leq \ell({\Sigma_0}) \Delta x\leq C' \Delta x.
\end{equation}
To prove the statement, we need an  estimate for the term $\int_{\Omega_j}|v(x)-W(x)|dx$ for every choice of $j$. With this aim, we remind that, for Theorem \ref{lipreg}, both $v(x)$ and $W(x)$ are Lipschitz continuous, so we can use a modification of the classical argument based on duplication of variables. Something similar can be found on \cite{S85,DE04} and \cite{S06}.\par
We are focusing on the problem (\ref{DE}) restricted on the region $\widehat{\Omega}_j:=\Omega_j \setminus \Sigma_{\Delta x}$ with some compatible Dirichlet conditions on $\Omega_j\cap\partial \Omega$. We do not have any Dirichlet conditions on $\partial\widehat{\Omega}_j\cap \partial \Sigma_{\Delta x}$, so we extend the boundary conditions as in \eqref{exd}. Inside the region $\widehat{\Omega}_j$ the solution $v(x)$ is Lipschitz continuous by Theorem \ref{lipreg}.\par

Let us choose a point $\widehat{x}\in G\cup \Omega_j=:G_j$ such that
\begin{equation}
   |v(\widehat{x})-W(\widehat{x})|=\max_{x \in G_{j}}  |v(x)-W(x)|
\end{equation}
and assume that $v(\widehat{x})\geq W(\widehat{x})$. The opposite case can be treated similarly. If $dist(\widehat{x},\partial\widehat{\Omega}_j)\leq\sqrt{h}$, implies, from the Dirichlet conditions and the Lipschitz continuity of $v$ and $W$ that 
\begin{equation}
   \max_{x \in G_j} |v(x)-W(x)|=v(\widehat{x})-W(\widehat{x})\leq C\sqrt{h}.
\end{equation}

Now, suppose that $dist(\widehat{x},\partial \widehat{\Omega}_j)>\sqrt{h}$ and define the auxiliary function 
\begin{equation}
\psi(x,y):=v(x)-W(y)-L_1\frac{|x-y-\sqrt{h} \eta|^2}{2 \sqrt{h}}-L_2\sqrt{h} {|y-\widehat{x}|^2}, \quad \hbox{ for }(x,y)\in\Omega_j\times G_j.
\end{equation}
Where $\eta$ is the inward normal to $\Omega_j$ like stated in previous assumptions.\par
It is not hard to check that the boundedness of $v$, $W$ and the continuity of $\psi$, imply the existence of some $(\overline{x},\overline{y})$ (depending on $h$) such that
\begin{equation}
  \psi(\overline{x},\overline{y})\geq \psi(x,y)\quad \hbox{ for all } (x,y)\in\widehat{\Omega}_j\times G_j.
\end{equation}
Since $dist(\widehat{x},\partial \widehat{\Omega}_j)>\sqrt{h}$, we have that $\widehat{x}+\sqrt{h}\eta\in\widehat{\Omega}_j$ and therefore
\begin{equation}
  \psi(\overline{x},\overline{y})\geq \psi(\widehat{x}+\sqrt{h}\eta,\widehat{x}),
\end{equation}
or equivalently
\begin{equation}\label{eqw}
v(\overline{x})-W(\overline{y})-\frac{L_1}{\sqrt{h}}|\overline{x}-\overline{y}-\sqrt{h}\eta|^2-L_2 \sqrt{h}|\overline{y}-\widehat{x}|^2\geq v(\widehat{x}-\sqrt{h}\eta)-W(\widehat{x}).
\end{equation}
\eqref{eqw} implies
\begin{multline}
\frac{L_1}{\sqrt{h}}|\overline{x}-\overline{y}-\sqrt{h}\eta|^2+L_2\sqrt{h}|\overline{y}-\widehat{x}|^2\leq v(\overline{x})-v(\widehat{x}-\sqrt{h}\eta)+W(\widehat{x})-W(\overline{y})\\
\leq v(\overline{x})-v(\overline{y})+\left[ \left( v(\overline{y}-W(\overline{y})\right)-\left( v(\widehat{x}-W(\widehat{x})\right) \right] +v(\widehat{x})-v(\widehat{x}-\sqrt{h}\eta)\\
\leq L_v |\overline{x}-\overline{y}|+\sqrt{h}\,L_v  
\leq  L_v |\overline{x}-\overline{y}-\sqrt{h}\eta|+2\sqrt{h}L_v  \\
\leq \frac{L_1}{2 \sqrt{h}}   |\overline{x}-\overline{y}-\sqrt{h}\eta|^2+\frac{\sqrt{h}}{2 L_1} \,L_v^2+2 \sqrt{h} \,L_v
\end{multline}
where $L_v$ is the Lipschitz constant of $v$, and therefore we can conclude
\begin{equation}\label{est1}
\frac{1}{h}|\overline{x}-\overline{y}-\sqrt{h}\eta|^2\leq \frac{1}{{L_1}^2}L_v^2+\frac{4}{L_1}L_v^2<\left(\frac{\epsilon}{2+\epsilon}\right)^2
\end{equation}

\begin{equation}\label{est2}
|\overline{y}-\widehat{x}|^2\leq \frac{1}{2 L_1 L_2} L_v^2+\frac{2}{L_2} L_v<\epsilon^2
\end{equation}
for a $\epsilon>0$, provided $L_1, L_2$ are sufficiently large.

Let us consider now the case $(\overline{x},\widehat{x})\in \widehat{\Omega}_j\times G_j$, so there are not on the boundary. \par
By (\ref{SDE}) we have, for a $x\in G_j$
\begin{equation}\label{2:pp}
  W\left(x-h\frac{\sum \tilde{a}^k \sigma_k(x)}{f(x)}\right)=W(x)+h  W(x)-h
\end{equation}
for some $\tilde{a}=\tilde{a}(x)$. This equation is verified a.e. and the point $x-h\frac{\sum \tilde{a}^k \sigma_k(x)}{f(x)}\in \Omega_j$ from the definition of the admissible choice of $\overline{a}$ and the hypothesis on the discretization steps. Since the map
\begin{equation}
x\mapsto v(x)-\left[W(\overline{y})+L_1\frac{|x-\overline{y}-\sqrt{h}\eta|^2}{2 \sqrt{h}}+L_2 \sqrt{h}{|\overline{y}-\widehat{x}|^2}\right]
\end{equation}
has a maximum at $\overline{x}$, by (\ref{EIK2}) we obtain
\begin{equation}\label{cross}
-L_1\frac{|(\overline{x}-\overline{y}-\sqrt{h}\eta)\cdot\sigma(\overline{x})|}{\sqrt{h}}\leq f_*(\overline{x})-f_*(\overline{x})v(\overline{x})
\end{equation}
and then
\begin{equation}\label{2:passs}
v(\overline{x})\leq 1+\frac{L_1}{f_*(\overline{x})}\frac{|(\overline{x}-\overline{y}-\sqrt{h}\eta)\cdot \sigma(\overline{x})|}{\sqrt{h}}\leq 1+\frac{L_1}{\sqrt{h}}(\overline{x}-\overline{y}-\sqrt{h}\eta)\cdot\frac{\sum \overline{a}^k \sigma_k(\overline{x})}{f_*(\overline{x})};
\end{equation}
the inequality $\psi(\overline{x},\overline{y})\geq \psi\left(\overline{x},\overline{y}-\frac{h}{f(\overline{y})}\sum \tilde{a}^k \sigma_k(\overline{y})\right)$ gives
\begin{multline}
-W(\overline{y})-L_1\frac{|\overline{x}-\overline{y}-\sqrt{h}\eta|^2}{2\sqrt{h}}-L_2  \sqrt{h}|\overline{y}-\widehat{x}|^2
\geq -W\left(\overline{y}-\frac{h}{f(\overline{y})}\sum \tilde{a}^k \sigma_k(\overline{y})\right)\\
-L_1\frac{\left|\overline{x}-h\overline{y}-\sqrt{h}\eta-\frac{\sum \tilde{a}^k \sigma_k(\overline{y})}{f(\overline{y})}\right|^2}{2 \sqrt{h}}-L_2 \sqrt{h} \left|\overline{y}-\widehat{x}-h\frac{\sum \tilde{a}^k \sigma_k(\overline{y})}{f(\overline{y})}\right|^2
\end{multline}
and then 
\begin{multline}
   W\left(\overline{y}-\frac{h}{f(\overline{y})}\sum \tilde{a}^k \sigma_k(\overline{y})\right)\\ \geq W(\overline{y})-\frac{L_1}{2\sqrt{h}}\left[ \left|\overline{x}-\overline{y}-\sqrt{h}\eta\right|^2 -   \left|\overline{x}-\overline{y}-\sqrt{h}\eta-\frac{\sum \tilde{a}^k \sigma_k(\overline{y})}{f(\overline{y})}\right|^2\right]\\+L_2 \sqrt{h}\left[ |\overline{y}-\widehat{x}|^2-\left|\overline{y}-\widehat{x}-\frac{\sum \tilde{a}^k \sigma_k(\overline{y})}{f(\overline{y})}\right|^2\right].
\end{multline}
Substituting the left hand side term with (\ref{2:pp}) and using the fact that for every $a,b,c\in\R^n$ we can prove that $|a-b|^2-|a-b-hc|^2=2h(a-b)\cdot c -h^2 |c|^2$,  we get
\begin{multline}
  W(\overline{y})\geq 1+\frac{L_1}{2 \sqrt{h^3}}\left[2h (\overline{x}-\overline{y}-\sqrt{h}\eta)\cdot \frac{ \sum \tilde{a}^k \sigma_k(\overline{y})}{f(\overline{y})}-h^2 \left| \frac{ \sum \tilde{a}^k \sigma_k(\overline{y})}{f(\overline{y})}\right|^2\right] \\
+\frac{L_2}{2\sqrt{h}}\left[2h (\overline{y}-\widehat{x})\cdot \frac{\sum \tilde{a}^k \sigma_k(\overline{y})}{f(\overline{y})}-h^2 \left| \frac{ \sum \tilde{a}^k \sigma_k (\overline{y})}{f(\overline{y})}\right|^2\right].
\end{multline}
Now, adding to (\ref{2:passs}) and using the estimations (\ref{est1}) and (\ref{est2})

\begin{multline}
 v(\overline{x})-W(\overline{y}) \leq \left(  \frac{L_1}{2} \sqrt{h}+\frac{L_2}{2}\sqrt{h^3}\right) \left|  \frac{\sum \tilde{a}^k \sigma_k(\overline{y})}{ f(\overline{y})}\right|^2
-\frac{L_1}{\sqrt{h}} (\overline{x}-\overline{y}-\sqrt{h}\eta)\\
\cdot \left(    \frac{\sum \tilde{a}^k \sigma_k(\overline{y})}{f(\overline{y})}- \frac{\sum \overline{a}^k \sigma_k(\overline{x})}{f_*(\overline{x})}\right) - L_2 \sqrt{h} (\overline{y}-\widehat{x})\cdot  \frac{\sum \overline{a}^k \sigma_k(\overline{x})}{f(\overline{x})}\\
\leq  \left(  \frac{L_1}{2} \sqrt{h}+\frac{L_2}{2}\sqrt{h^3}\right) \left|  \frac{\sum \tilde{a}^k \sigma_k(\overline{y})}{ f(\overline{y})}\right|^2\\-L_1 \frac{\epsilon}{2+\epsilon}  \left|   \frac{\sum \tilde{a}^k \sigma_k(\overline{y})}{f(\overline{y})}- \frac{\sum \overline{a}^k \sigma_k(\overline{x})}{f_*(\overline{x})}\right| 
 -L_2 \sqrt{h}\epsilon \left|  \frac{\sum \overline{a}^k \sigma_k(\overline{x})}{f(\overline{x})} \right|.
\end{multline}
\noindent Finally, choosing $\epsilon=\sqrt{h}$  by the boundedness of $f$ and $\sigma$, we obtain
\begin{equation}
v(\overline{x})-W(\overline{y})\leq C \sqrt{h} 
\end{equation}
where  $C$ is a  suitable positive constants.
Then the inequality $\psi(\overline{x},\overline{y})\geq \psi(x,x)$ yields
\begin{equation}\label{2:pasfin1}
v(x)-W(x)\leq v(\overline{x})-W(\overline{y})\leq C \sqrt{h}
\end{equation}
for all $x\in\widehat{\Omega}_j$. 

Finally we consider the case when $\overline{y}\in \partial G_j$ or $\overline{x}\in \partial \widehat{\Omega}_j$. If $\overline{y}\in \partial G_j$.  the Dirichlet conditions imply that $v(\overline{y})=W(\overline{y})$ and we have 
\begin{multline}
v(\widehat{x})-W(\widehat{x})\leq v(\widehat{x}-\sqrt{h}\eta)-v(\widehat{x})+v(\overline{y})-v(\overline{x})\\
\leq L_v (\sqrt{h}+|\overline{x}-\overline{y}|)\leq L_v (2\sqrt{h}+|\overline{x}-\overline{y}-\sqrt{h}\eta|)\leq C\sqrt{h}.
\end{multline}
In a similar way we can treat the case $\overline{x}\in\partial \widehat{\Omega}_j$.\par
To prove the inequality $W(x)-v(x)\leq C \sqrt{h}$ it is enough to interchange the roles of $v$ and $W$ on the auxiliary function $\psi$.\par
We add this estimation  in (\ref{pass1}), getting the thesis 
\begin{equation}
||v(x)-W(x)||_{L^1}\leq C\sqrt{h}+ C'\Delta x.
\end{equation}

\end{proof}

\section{Numerical experiments and applications} \label{sec:tests} 

In this section we present some results for (\ref{EK}) on some test problems coming from front propagation, control theory and image processing. In all these examples the discontinuity of the coefficients appears in a natural way and has an easy interpretation with respect to the model.

\subsection{Test 1: a front propagation problem}

Front propagation problems arise in a lot of different fields of mathematics. A typical approach is to use the Hamilton-Jacobi framework to solve them, as in \cite{OS88}. Our first test can be interpreted as a front propagation in a discountinuous media. In this model, the level sets of the value function have the meaning of the regions with the same time of arrival of the front. \par
Let $\Omega:=(-1,1)\times(0,2)$ and $f:\Omega \rightarrow \R$ be defined by

\begin{figure}[tb]
\begin{center}
\includegraphics[height=6cm]{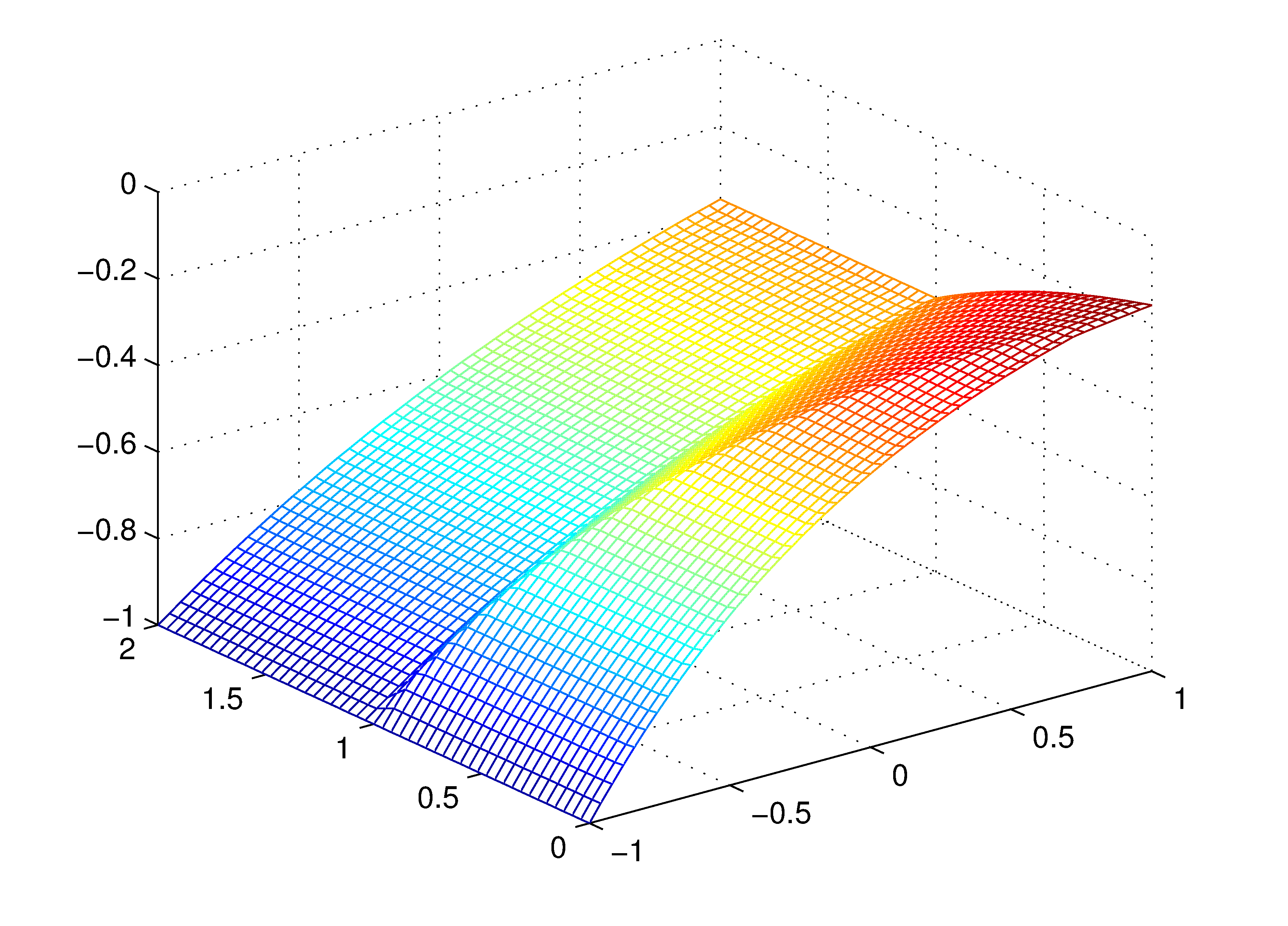}  \\
\includegraphics[height=6cm]{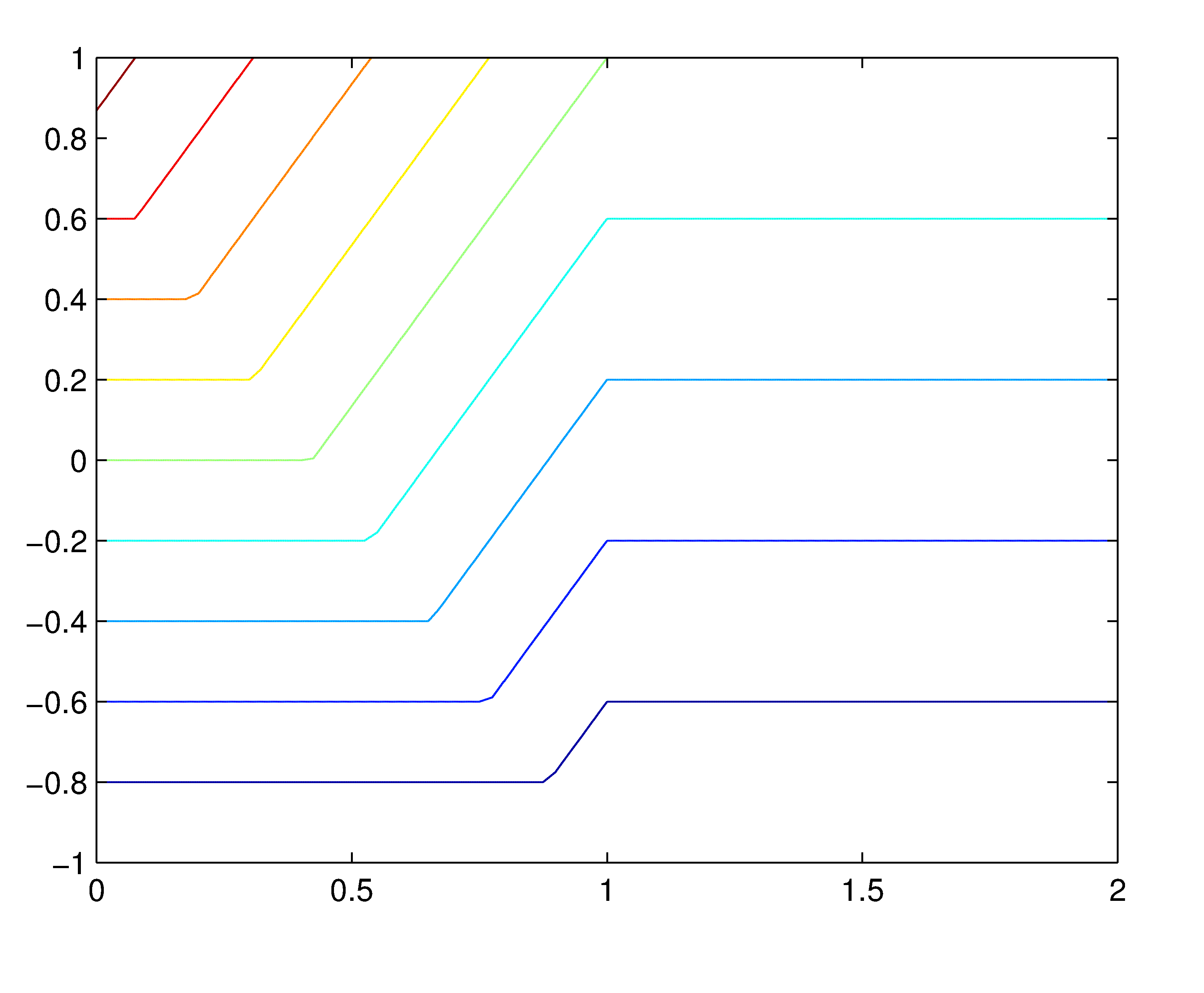}  \\
\caption{Test 1.} \label{t:3}
\end{center}
\end{figure}

\begin{equation}\label{test1}
  f(x_1,x_2):=\left\{ 
  \begin{array}{ll}
  1 & x_1<0,\\
  3/4 & x_1=0 \\
  1/2 & x_1>0
  \end{array}
  \right.
\end{equation}
It is not difficult to see that $f$ satisfies conditions  (\ref{nocusp}). We can verify that the function

\begin{equation}
  u(x_1,x_2):=\left\{ 
  \begin{array}{ll}
  \frac{1}{2} x_2, & x_1\geq 0,\\
  \\
    -\frac{\sqrt{3}}{2}x_1+\frac{1}{2}x_2, & -\frac{1}{\sqrt{3}}x_2\leq x_1 \leq 0,\\
    \\
  x_2,& x_1< -\frac{1}{\sqrt{3}}x_2.
  \end{array}
  \right.
\end{equation}
is  a viscosity solution of $|Du|=f(x)$ in the sense of our definition. Moreover, we take $g:= u_{|\partial \Omega}$. We show in the Table \ref{tt:3} and in Figure \ref{t:3} our results.\par
\begin{table}
\begin{center}
\begin{tabular}{||p{2.3cm}||*{4}{c|}|}
\hline
         $\Delta x=h $      &   $||\cdot||_\infty$  &$Ord(L_\infty)$&  $||\cdot||_1$ &  $Ord(L_1)$ \\
\hline
\hline
\bfseries 0.1  & 1.734e-1  &     &  8.112e-2&  \\
\hline
\bfseries 0.05 & 8.039e-2  &  1.1095 &  3.261e-2  &  1.3148\\
\hline 
\bfseries 0.025  &  4.359e-2 &  0.8830   &   1.616e-2 &  1.0178\\
\hline
\bfseries 0.0125  &  2.255e-2   &    0.9509   &   7.985e-3  &   1.0271\\
\hline
\end{tabular}

\caption{Test 1: experimental error.}\label{tt:3}
\end{center}

\end{table}

We also show, in Table \ref{ttc:4} a comparison between the FD methods proposed in \cite{DE04}. They proposed two techniques: in the first there is a regularization of the Hamiltonian with a viscosity term ($DF-reg$), in the second one ($DF-FS$), better results are obtained, but numerically there are more difficulties; the authors solve them using $Fast Sweeping$ (see \cite{Z04}) as acceleration technique and they archive very good results. Our technique has, in this test, a performance similar to $DF-reg$, in our scheme, the interpolation operator (in this case bilinear) adds a regularization. We aspect better performances of our method rather DF techniques on more complicated cases, where characteristics are not straight lines.

\begin{table}
\begin{center}
\begin{tabular}{||p{2.3cm}||*{6}{c|}|}
\hline
         $\Delta x=h $      &   our method  &$Ord$&  DF-reg &  $Ord$ &  DF-FS &  $Ord$\\
\hline
\hline
\bfseries 0.1  & 1.734e-1  &     &  1.243e-1&      &  5.590e-2& \\
\hline
\bfseries 0.05 & 8.039e-2  &  1.1095 &  7.229e-2  &  0.78    &  2.795e-2&  1.00\\
\hline 
\bfseries 0.025  &  4.359e-2 &  0.8830   &   4.085e-2 &  0.82     &  1.397e-2& 1.00 \\
\hline
\bfseries 0.0125  &  2.255e-2   &    0.9509   &   2.266e-2  &   0.85    &  3.493e-3& 1.00 \\
\hline
\end{tabular}

\caption{Test 1: comparison between different numerical methods (uniform norm).}\label{ttc:4}
\end{center}
\end{table}

\subsection{Test 2: a control problem with a discontinuous value function}

\begin{figure}[tb]
\begin{center}
\includegraphics[height=6cm]{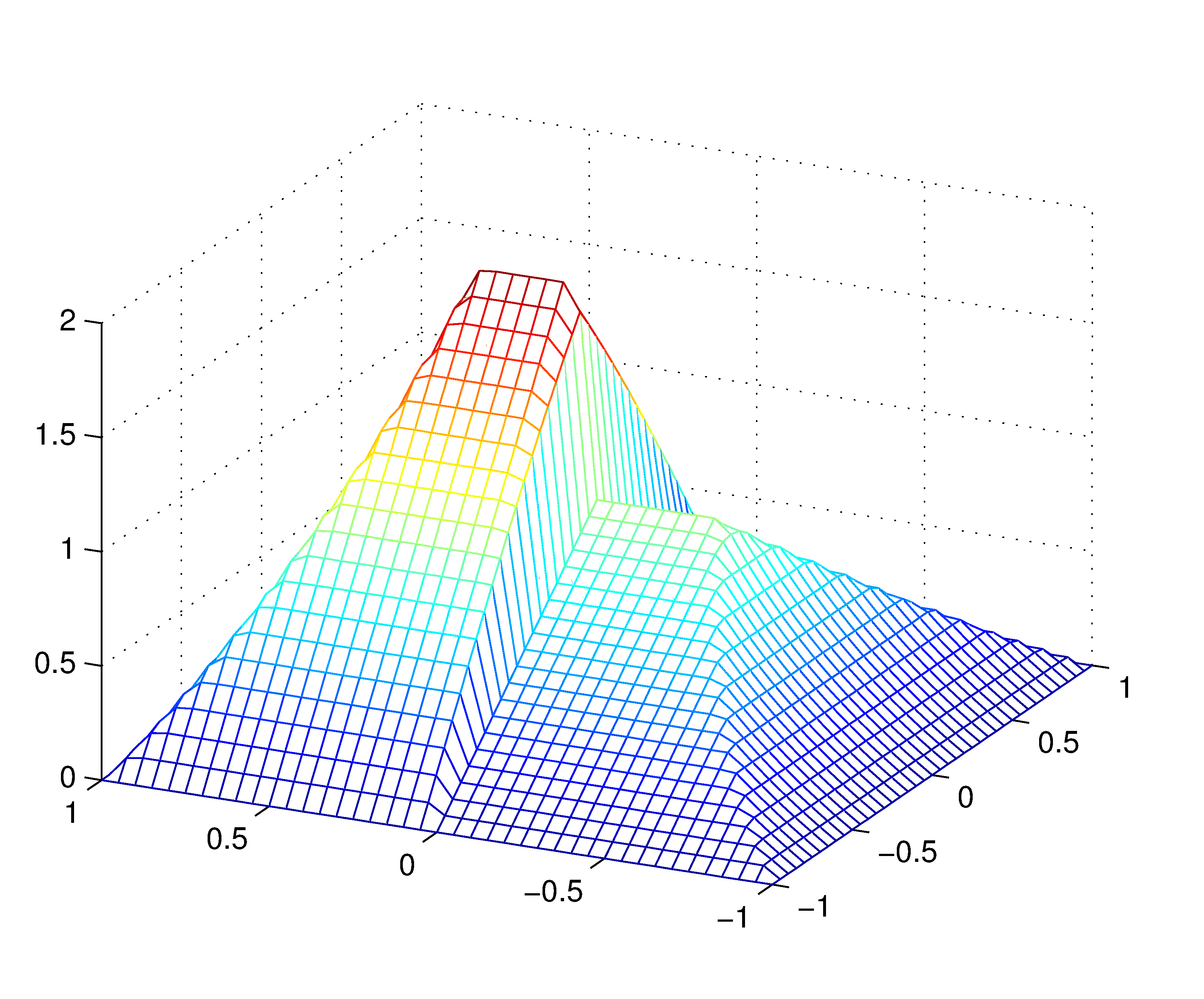}  \\
\includegraphics[height=6cm]{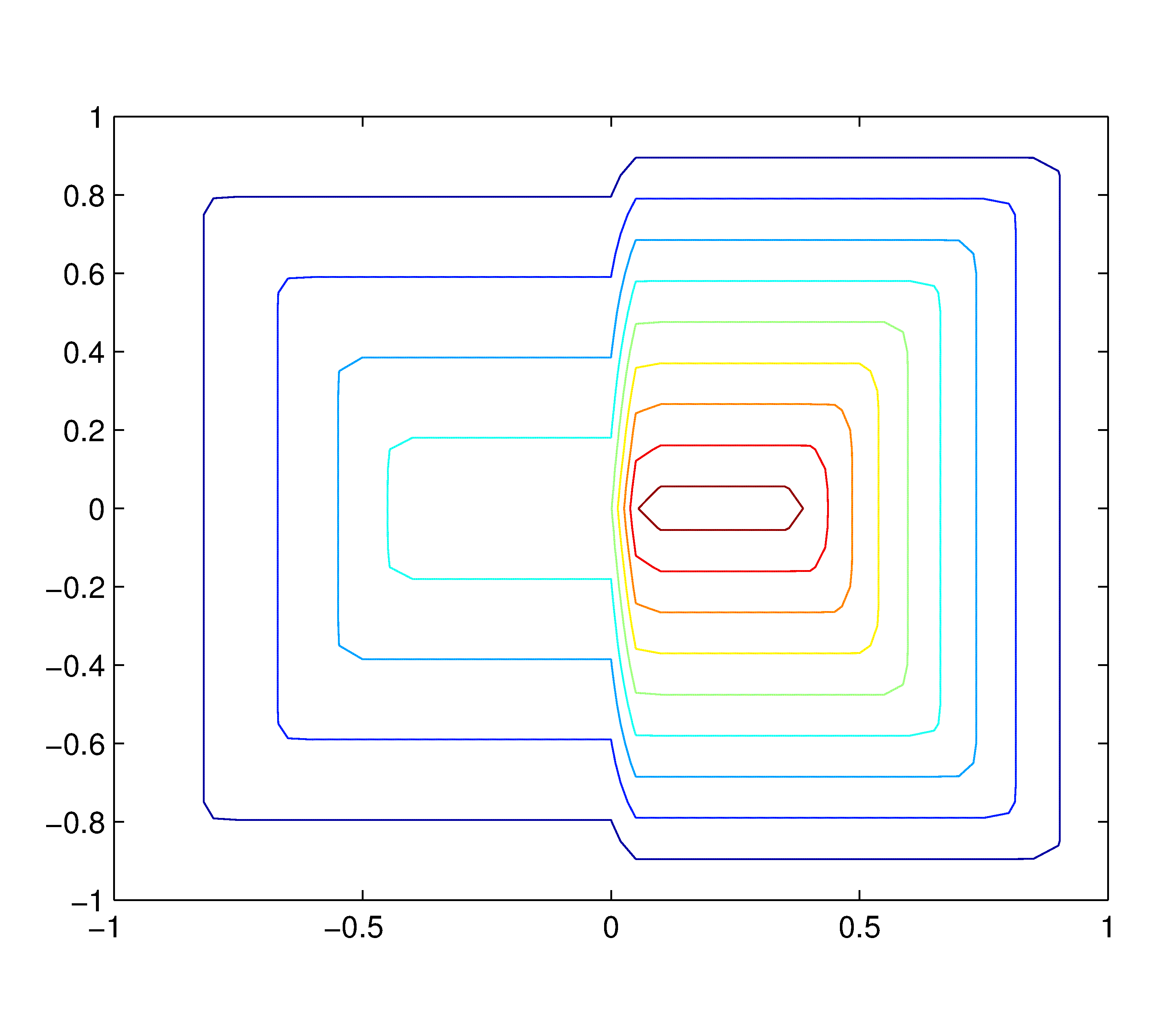}  \\
\caption{Test 2.} \label{t:5}
\end{center}
\end{figure}

In this test we present a case where a continuous solution does not exist. In this case it is evident that a convergence in uniform norm will not be possible. \par
  We consider the problem shown in the example \ref{ex2}. As already said, let $\Omega:=[-1,1]^2$ we want to solve 
\begin{equation}
\left\{
\begin{array}{cc}
x^2\left(u_x(x,y)\right)^2+\left(u_y(x,y)\right)^2=\left[f(x,y)\right]^2 & ]-1,1[\times]-1,1[\\
u(\pm 1,y)=u(x,\pm 1)=0 & x,y\in[-1,1]
\end{array}
\right.
\end{equation}
with $f(x, y) = 2$, for $x > 0$, and $f(x, y) = 1$ for $x \leq 0$. The correct viscosity solution is

\begin{equation}
u(x,y)=\left\{
\begin{array}{cc}
2(1-|y|) & x> 0, |y|>1+\ln x\\
-2ln(x) & x> 0, |y|\leq 1+\ln x\\
\frac{u(-x,y)}{2} & x\leq 0.
\end{array}
\right.
\end{equation}

We show in Figure \ref{t:5} our results. In this case the convergence in the uniform norm fails. Convergence in the integral norm $L^1$ as proved in Section \ref{sec:conv} is confirmed by Table \ref{tt:5}.

\begin{table}
\begin{center}
\begin{tabular}{||p{3.3cm}||*{4}{c|}|}
\hline
         $\Delta x=h $      &   $||\cdot||_\infty$  &$Ord(L_\infty)$&  $||\cdot||_1$ &  $Ord(L_1)$ \\
\hline
\hline
\bfseries 0.2  & 1.0884  &     &  0.4498&  \\
\hline
\bfseries 0.1  & 1.0469  &   -  &  0.2444& 0.88 \\
\hline
\bfseries 0.05 & 1.0242  &  - &  0.1270  &  0.9444\\
\hline 
\bfseries 0.025  &  1.0123 &  -   &   0.0628 &  0.9708\\
\hline
\bfseries 0.0125  &  1.0062  &    -   &   0.0327  &   0.9867\\
\hline
\bfseries 0.00625  &  1.0031  &    -   &   0.0221  &   0.5652\\
\hline
\end{tabular}

\caption{Test 2: experimental error.}\label{tt:5}
\end{center}
\end{table}

\subsection{Test 3: Shape-from-Shading with discontinuous brightness}

The Shape-from-Shading problem consists in reconstructing the three dimensional shape
of a scene from the brightness variation (shading) in a greylevel photograph of that scene.
The study of the Shape-from-Shading problem started in the 70s (see \cite{HB89} and references
therein) and since then a huge number of papers have appeared on this subject. More
recently, the mathematical community was interested in Shape-from-Shading since its
formulation is based on a first order partial differential equation of Hamilton-Jacobi type (see \cite{RT92,O99}).\par
The equation related to this problem is the following: for  a brightness function $I(x,y):\R^2\supset\Omega \rightarrow [0,1]$, in the case of 
vertical light source is vertical, to reconstruct the unknown surface, we need to solve
\begin{equation}\label{5:sfseik}
  |Du(x,y)|= \left(\sqrt{\frac{1}{I(x,y)^2}-1}\right), \quad (x,y)\in \Omega.
\end{equation}
Points $(x,y)$ where $I$ is maximal (i.e. equal to 1) correspond to the particular situation
when the light direction and $n$ are parallel. These points are usually called ``singular
points'' and, if they exist, equation \eqref{5:sfseik} is said to be degenerate. The
notion of singular points is strictly related to that of concave/convex ambiguity, we refer to \cite{LRT93,IR95} for details on this point.\par

It is important to note that, whatever the final equation is, in order to compute a solution
we will have to impose some boundary conditions on $\partial \Omega$ and/or inside $\Omega$. A natural choice
is to consider Dirichlet type boundary conditions in order to take into account at least two
different possibilities. The first corresponds to the assumption that the surface is standing
on a flat background, i.e. we set $u(x,y) = 0$  for $(x,y)\in \partial \Omega$. The second possibility occurs when the height of the surface on the boundary ({\em silhouette}) is known: $u(x,y) = g(x,y)$  for $ (x,y)\in \partial \Omega$.
The above boundary conditions are widely used in the literature although they are often
unrealistic since they assume a previous knowledge of the surface.\par
Let us focus on two important points: 
\begin{itemize}
\item We note that a digital image is always a discontinuous datum. Is is a piecewise constant function with a fixed measure of his domain of regularity (pixel). So this is the interest of our analysis for discontinuous cases of $f$.
\item In the case of maximal gray tone ($I(x)=1$), we are not in the Hypothesis introduced previously. In particular we have that $f=0$ in some points. We overcome this difficulty, as suggest in \cite{CG00}. We regularize the problem making a truncation of $f$. It is possible to show that this regularized problem goes to the maximal subsolution of the problem with $\epsilon \rightarrow 0^+$. And that this particular solution is the correct one from the applicative point of view.\par
\end{itemize}

We consider, now a test with a precise discontinuity on $I$, and we will discuss some issue about this case.\par
We firstly consider a simple problem in 1D to point out an aspect of the model. Let the function $I$ be 
\begin{equation}
I=\left\{ \begin{array}{ll}
\sqrt{1-x^2} & \hbox{if } -1 \leq x \leq 0.2\\
\frac{\sqrt{2}}{2} & \hbox{if } 0.2 \leq x \leq 1\\
0 & \hbox{otherwise }
\end{array} \right.
\end{equation}
we can see that we have a discontinuity on $x=0.2$; despite this, because of the non degeneracy of the dynamics, the solution will be continuous. For this reason we can see that changing the boundary condition of the problem, the solution will be the maximal Lipschitz solution that verifies continuously the boundary condition. 

\begin{figure}[ht]
\begin{center}
\includegraphics[width=8cm]{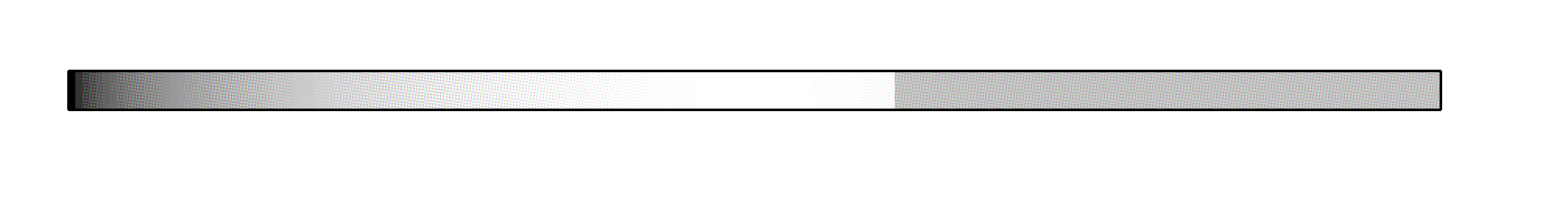}  \\
\includegraphics[width=8cm]{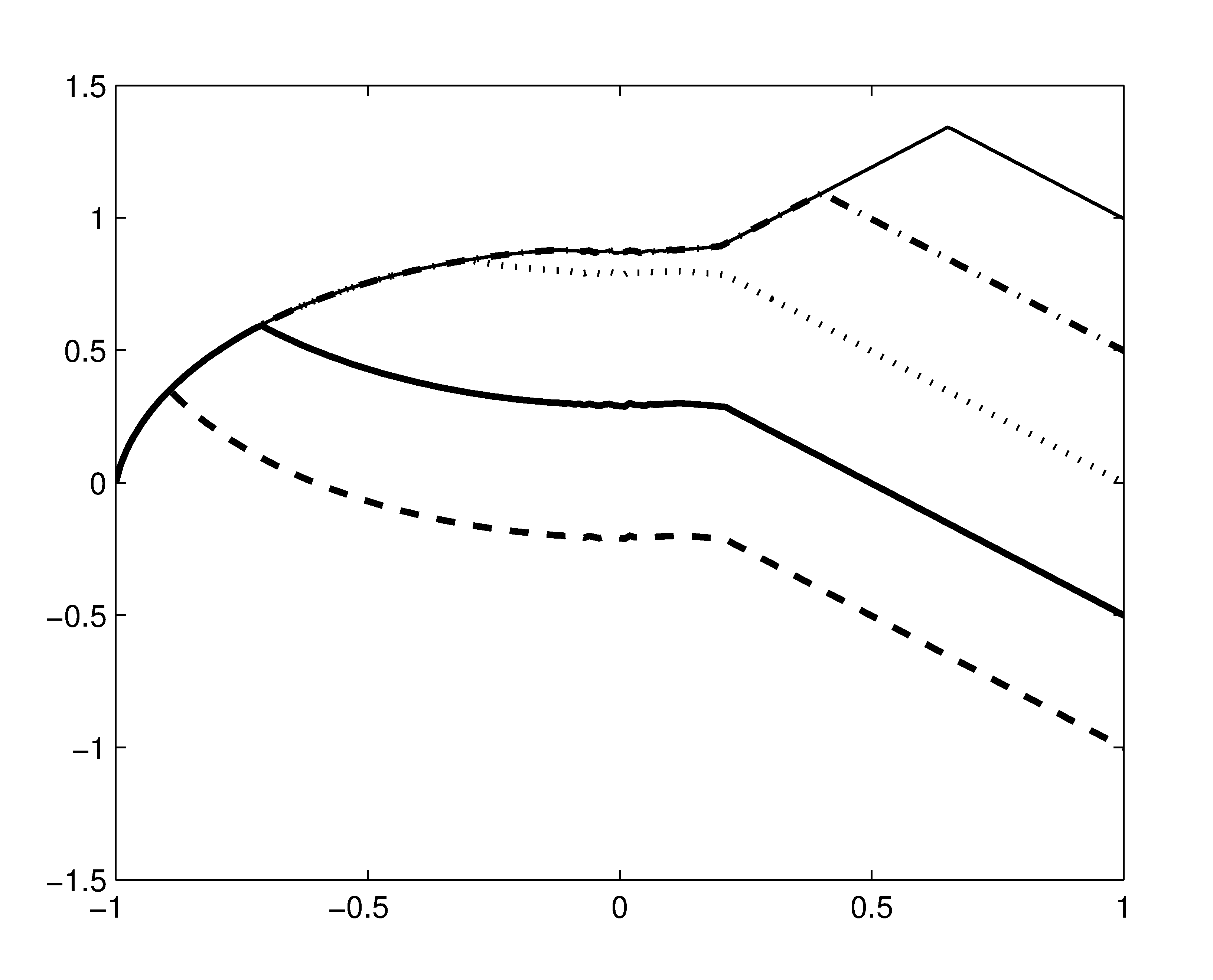}  \\
\caption{Sfs-data and solution with various boundary values.} \label{5:testbound}
\end{center}
\end{figure}

To see this we have solved this simple monodimensional problem with various Dirichlet condition, in particular we require $u(-1)=0$, and $u(1)=\{-1,0.5,0,0.5,1\}$. With $\Delta x=0.01$ and $\Delta t=0.002$, we obtain the results shown in Figure \ref{5:testbound}.\par
We can realize, in this way, an intrinsic limit of the model. It can not represent an object with discontinuities. We make another example that is more complicated and more close to a real application. \par
\begin{figure}[ht]
\begin{center}
\includegraphics[height=5.5cm]{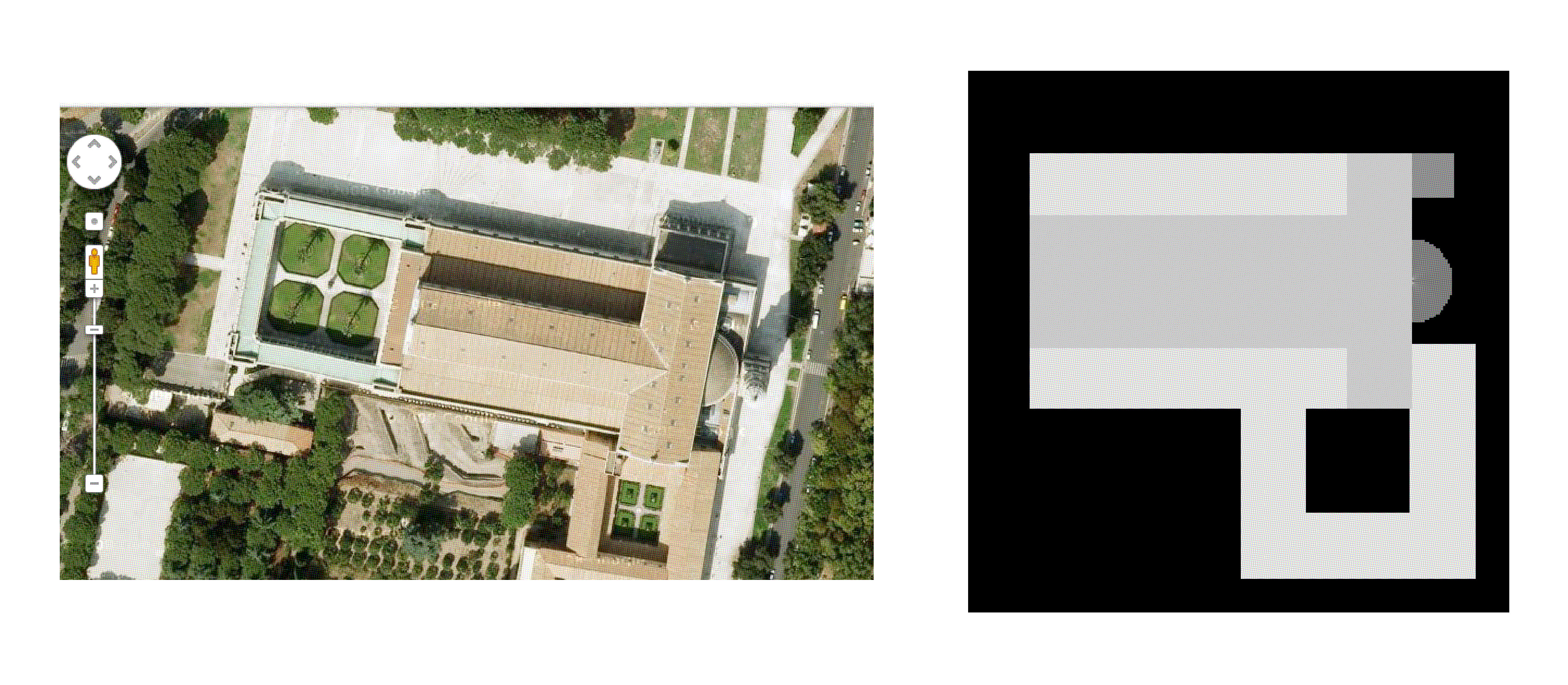}  \\
\caption{Basilica of Saint Paul Outside the Walls: satellite image and simplified sfs-datum.} \label{5:Sanpaolo}
\end{center}
\end{figure}
We consider a simplified sfs-datum for the Basilica of Saint Paul Outside the Walls in Rome, as shown in Figure \ref{5:Sanpaolo}. We have not the correct boundary value on the silhouette of the image and on the discontinuities, so we impose simply $u\equiv 0$ on the boundary. Computing the equation with $\Delta t=0.001$ we get the solution described on Figure \ref{5:sp1}.
\begin{figure}[ht]
\begin{center}

\includegraphics[height=8cm]{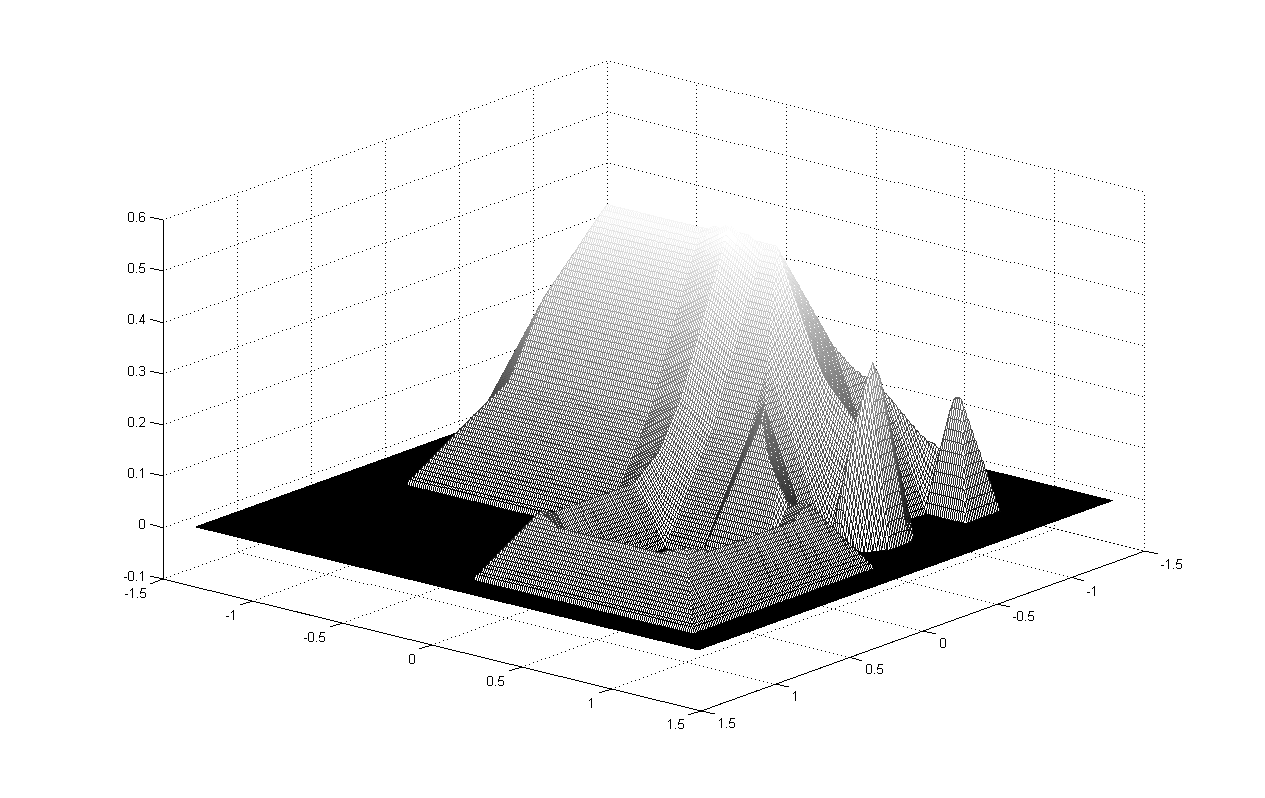}  \\
\caption{Test 3: reconstructed shape without boundary data.} \label{5:sp1}
\end{center}
\end{figure}

We can see that, although the main features of the shape as the slope of the roofs, the points of maximum are well reconstructed. Despite it, the shape which we get is not so close to our expectations. We can try to get better results adding the correct height of the surface along the silhouette as discussed above and, in this case, we get the solution shown on Figure \ref{5:sp2}. We can notice a more convincing shape, but also in this case it is quite not satisfactory. For example we have that the correct boundary conditions we imposed are not attained, and we create some discontinuity on some parts of them. This is due to the fact that they can be not compatible with the statement of the problem. Essentially the limit which we can see, as described above, is that we cannot have discontinuity on the viscosity solution (Theorem \ref{lipreg}). \par

We propose a different model for this problem, which allows discontinuous solutions. At this point we do not care about the physical interpretation of it, instead we are trying to find a solution closer to the correct solution.  We want to solve the equation
\begin{equation}
\left\{ \begin{array}{ll}
\max\limits_{|a|\leq 1}\left\{-Du(x)\cdot \sum\limits_{k=1}^2 a^k \sigma_k(x,y)\right\}=\sqrt{\frac{1}{I^2(x,y)}-1} & x\in\Omega\\
u(x)=g(x) & x\in\partial\Omega
\end{array} \right.
\end{equation}
with the map $\sigma:\Omega\rightarrow \R^{2,2}$ is
\begin{equation}
 \left(\begin{array}{cc}\left(1+\left|I(x-h,y)-I(x+h,y)\right|\right)^{-p} & 0 \\  0 & \left(1+\left|I(x,y-h)-I(x,y+h)\right|\right)^{-p}\end{array}\right)
\end{equation}
where $p\in \R$ is a tuning parameter. Obviously this choice of the anisotropic evaluator $\sigma$ is a bit trivial. This pick is done for the sake of simplicity. More complicated proposal can be found for example in \cite{AK01}.\par
In this way we use the results about the degeneracy of the dynamics permitting to the viscosity solution to be discontinuous. Of course this is, in some sense, the opposite situation with respect to the classical formulation: in this case every non smooth point of the surface is interpreted as discontinuity and we try to reconstruct it using the data coming from the silhouette.\par
The results are shown in Figure \ref{5:sp3} and in Table \ref{tt:6}
 we can see an accuracy comparison of   the various procedure.

\begin{figure}[ht]
\begin{center}
\includegraphics[height=7cm]{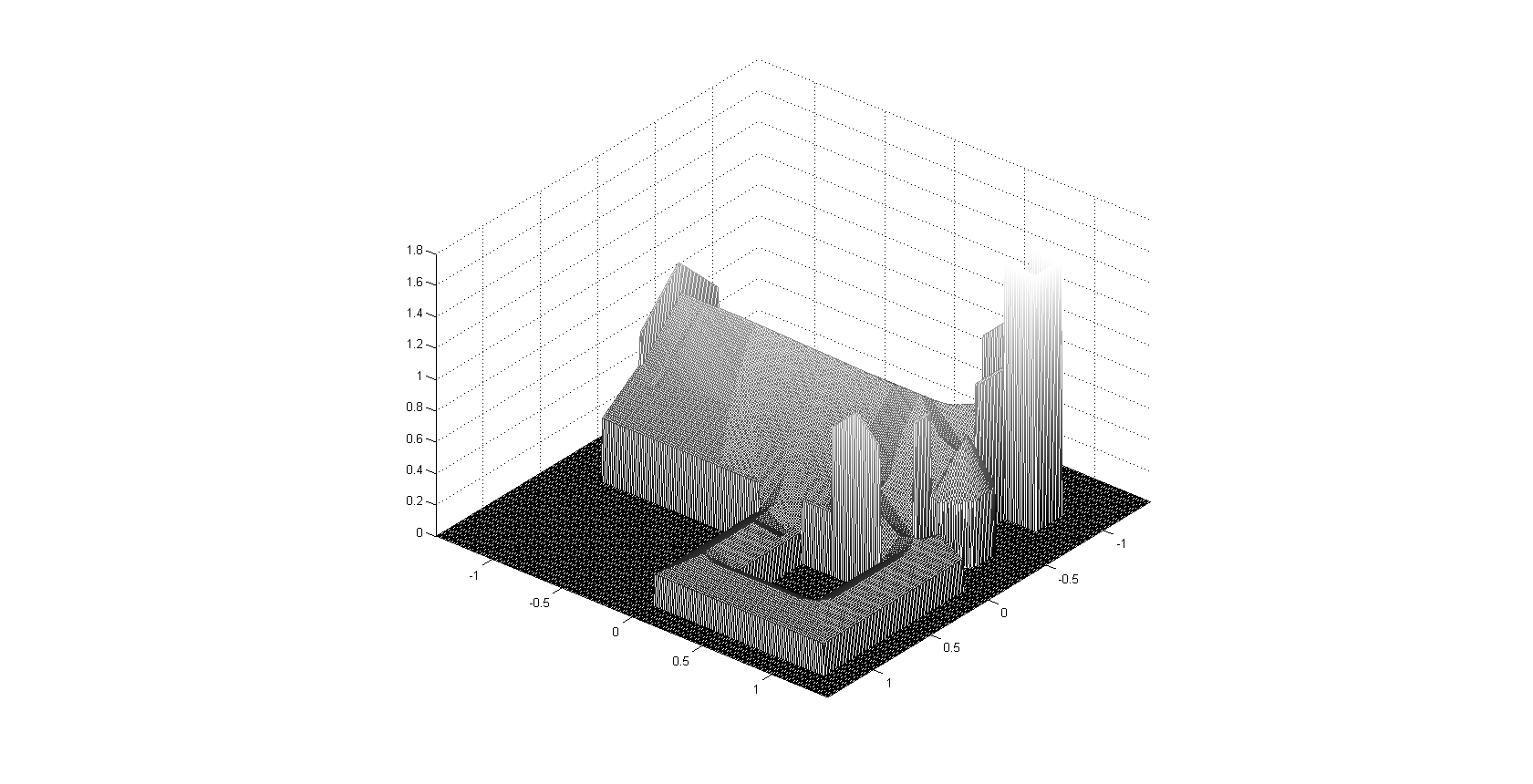}  \\
\caption{Test 3: Dirichlet condition on the silhouette.}   \label{5:sp2}
\includegraphics[height=7cm]{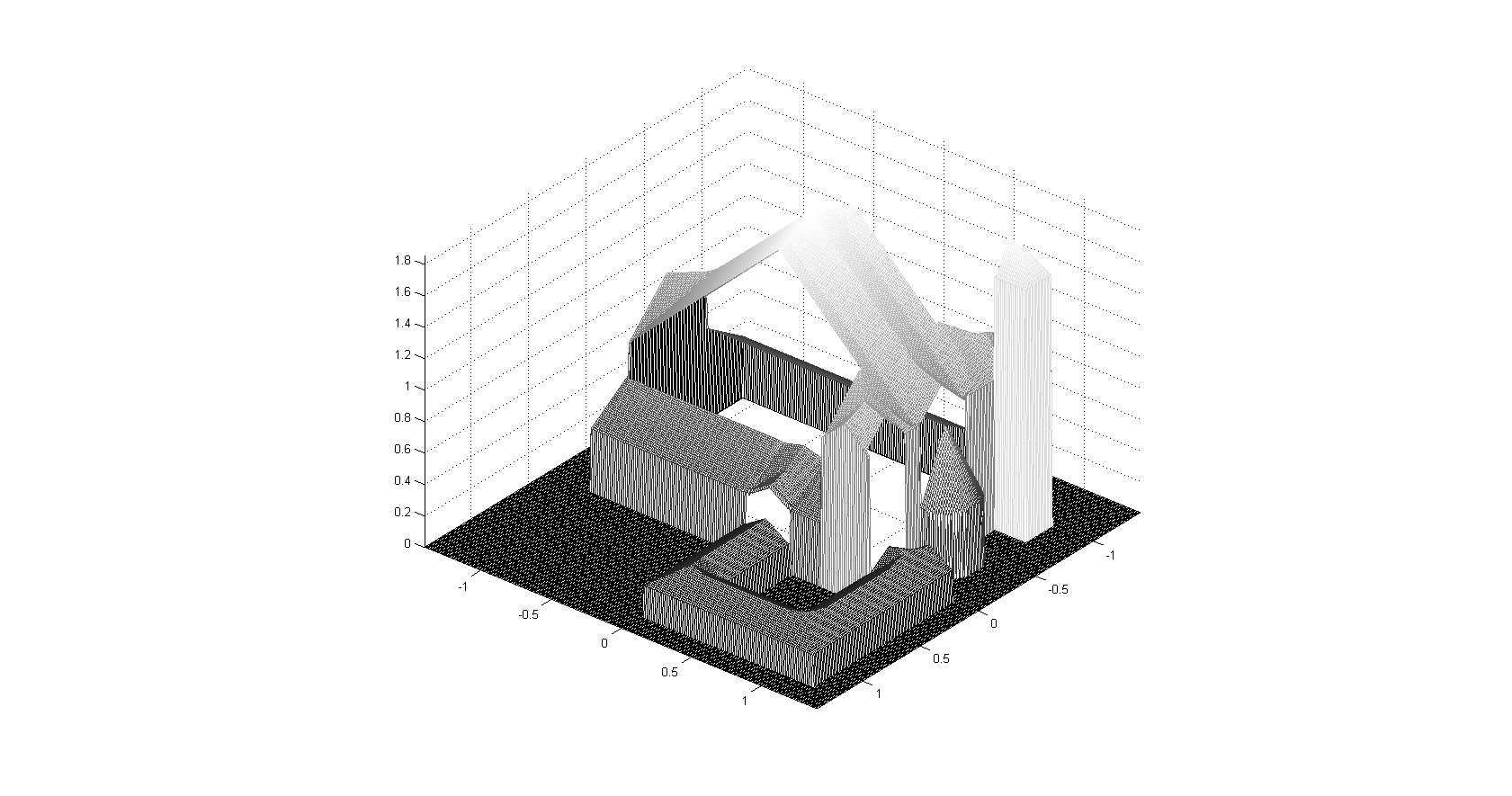}  \\
\caption{Test 3: Dirichlet condition and discontinuous dynamics.} \label{5:sp3}
\end{center}
\end{figure}

\begin{table}
\begin{center}
\begin{tabular}{||p{6.3cm}||*{2}{c|}|}
\hline
      Test      &  $||\cdot||_\infty$ & $||\cdot||_{L^1}$ \\
\hline
\hline
\bfseries w/o correct boundary data & 1.7831  &    1.5818 \\
\hline
\bfseries w boundary data  & 0.8705  &  0.5617 \\
\hline
\bfseries w boundary + disc. detect. & 0.7901 &  0.3062\\
\hline 

\end{tabular}

\caption{Test 3: Comparison between various methods} \label{tt:6}
\end{center}
\end{table}


\begin{thebibliography}{00}
\bibitem{AK01}{\sc G. Aubert and  P. Kornprobst}, {\em Mathematical Problems in Image Processing: Partial Differential Equations and the Calculus of Variations, } Springer Verlag, Applied Mathematical Sciences, Vol 147, 2001.
\bibitem{BCD97} {\sc M. Bardi and I. Capuzzo-Dolcetta}, {\em Optimal Control and Viscosity Solution of Hamilton-Jacobi-Bellman Equations}. Birkhauser, Boston Heidelberg, 1997.
\bibitem{B98} {\sc G. Barles}, {\em Solutions de viscosit\`e des equations d'Hamilton--Jacobi}, Springer--Verlag, 1998.
\bibitem{B93}{\sc G. Barles}, {\em Discontinuous viscosity solutions of first-order Hamilton-Jacobi equations: a guided visit}. J. Nonlin. Anal. 20 9, (1993), pp. 1123--1134. 
\bibitem{BS91} {\sc G. Barles and  P.E. Souganidis}, {\em Convergence of approximation schemes for fully nonlinear second order equations}, Asympt. Anal., {4} (1991), pp. 271--283.

\bibitem{BJ90} {\sc E. N. Barron and R. Jensen},  {\em Semicontinuous viscosity solutions for Hamilton-Jacobi equations with convex Hamiltonians}, Comm. Partial. Diff. Eq., 15 (1990), pp. 1713--1742.
\bibitem{BFZ10}{\sc O. Bokanowsky, N. Forcadel and H. Zidani},  {\em L1-error estimates for numerical approximations of Hamilton-Jacobi-Bellman equations in dimension 1}. Math. of Comput., 79, (2010), pp.  1395--1426.
\bibitem{BFZ10a} {\sc O. Bokanowski, N. Forcadel and H. Zidani}, {\em Reachability and minimal times for state constrained nonlinear problems without any controllability assumption},  SIAM J. Control Optim., {48} (2010),  pp. 4292--4316.

\bibitem{CL84} {\sc M.G. Crandall and P.L. Lions}, {\em Two approximations of solutions  of Hamilton--Jacobi equations}, Math. Comp., { 43} (1984), pp. 1--19. 
\bibitem{CG00} {\sc F. Camilli and L.Gr\"une}, {\em Numerical approximation of the maximal solution of a class of degenerate Hamilton-Jacobi equations}, SIAM J. Numer. Anal. 38 (2000), pp. 1540--1560.
\bibitem{CF07} {\sc E. Cristiani and M. Falcone}, {\em Fast Semi-Lagrangian Schemes for the Eikonal Equation and Applications}, SIAM J. Numer. Anal., 45 5 (2007),  pp. 1979--2011.
\bibitem{DMF00}{\sc G. Dal Maso and H. Frankowska}, {\em Value functions for Bolza problems with discontinuous Lagrangians and Hamilton-Jacobi inequalities}, ESAIM Control Optim. Calc. Var., 5 (2000), pp. 369--393.
\bibitem{DE04} {\sc K. Deckelnick and C. Elliott}, {\em Uniqueness and error analysis for Hamilton-Jacobi equations with discontinuities}, Interface. free bound., 6 (2004), pp. 329--349.
\bibitem{FF02} {\sc M. Falcone and R. Ferretti}, {\em Semi-Lagrangian schemes for Hamilton-Jacobi equations, discrete rapresentation formulae and Gordunov methods}, J. Comput. Phys., 175 (2002), pp. 559--575.
\bibitem{FF13} {\sc M. Falcone and R. Ferretti}, {\em Semi-Lagrangian Approximation Schemes for Linear and Hamilton-Jacobi Equations}, SIAM, 2013.
\bibitem{F93}{\sc H. Frankowska}, {\em Lower semicontinuous solutions of Hamilton-Jacobi-Bellman equations}, SIAM J. Control Optim., 31 (1993), pp. 257--272.
\bibitem{HB89}{\sc B. K. P. Horn and M. J. Brooks}, {\em Shape from Shading}, MIT Press, 1989.
\bibitem{I85} {\sc H. Ishii}, {\em  Hamilton-Jacobi equations with discontinuous Hamiltonians an arbitrary open sets},  Bull. Fac. Sci. Engrg. Chuo. Univ., 28  (1985), pp. 33--77.
\bibitem{I87} {\sc H. Ishii}, {\em A simple, direct proof of uniqueness for solutions of the Hamilton-Jacobi equations of eikonal type}, Proc. Amer. Math. Soc., 100  2 (1987), pp. 247--251.
\bibitem{IR95}{\sc H. Ishii and M. Ramaswamy}, {\em Uniqueness results for a class of Hamilton-Jacobi equations with singular coefficients}, Comm. Partial Differential Equations, 20 (1995), pp. 2187--2213. 
\bibitem{LT01} {\sc C.T. Lin and  E. Tadmor}, {\em $L^1$ stability and error estimates for Hamilton--Jacobi solutions}, Num. Math., { 87} (2001), 701--735.
\bibitem{LRT93}{\sc P.L. Lions, E. Rouy and A. Tourin}, {\em Shape from shading, viscosity solution and edges}, Num. Math., 64 (1993), pp. 323--353.
\bibitem{NS95} {\sc R.T. Newcomb and J. Su}, {\em Eikonal equations with discontinuities}, Diff. Integral Equations, 8 (1995), pp. 1947--1960.
\bibitem{OS88}{\sc S. Osher and J.A.  Sethian}, {\emph Fronts propagating with curvature-dependent speed: algorithms based on Hamilton-Jacobi formulations.}, J. Comput. Phys., 79 (1988), pp. 12--49.
\bibitem{O99}{\sc D. Ostrov}, {\em Viscosity solutions and convergence of monotone schemes for synthetic aperture radar shape-from-shading equations with discontinuous intensities}, SIAM J. Appl. Math., 59 (1999), pp. 2060--2085.
\bibitem{RT92}{\sc E. Rouy and A. Tourin}, {\em A viscosity solutions approach to shape-from-shading}, \emph{SIAM J. Numer. Anal.} 29 (1992), 867--884.
\bibitem{S02} {\sc P. Soravia}, {\em Boundary Value Problems for Hamilton-Jacobi Equations with Discontinuous Lagrangian}, Indiana Univ. Math. J., 51 (2002), pp. 451--77.
\bibitem{S06} {\sc P. Soravia}, {\em Degenerate eikonal equations with discontinuous refraction index},  ESAIM Control Optim. Calc. Var., 12  2 (2006), pp. 216--230.
\bibitem{S85}{\sc P.E. Souganidis}, {\em  Approximation schemes for viscosity solutions of Hamilton-Jacobi equations}, J. Differential Equations, 57 (1985), pp. 1--43.
\bibitem{T91} {\sc E. Tadmor}, {\em Local error estimates for discontinuous solutions of nonlinear hyperbolic equations}, SIAM J. Num. Anal., { 28} (1991), pp. 891--906.
\bibitem{TGO01}{\sc Y.R. Tsai, Y. Giga and S. Osher}, {\em A Level Set Approach for Computing Discontinuous Solutions of a Class of Hamilton-Jacobi Equations}, Math. Comp. 72 (2001), pp. 159--181.
\bibitem{T92} {\sc A. Turin}, {\em A comparison theorem for a piecewise Lipschitz continuous Hamiltonian and applications to shape-from-shading}, Numer. Math., 62 (1992), pp. 75--85.
\bibitem{Z04}{\sc H. Zhao}, {\em A Fast Sweeping Method for Eikonal Equations}, Math. Comp., 74 250 (2004), pp. 603--627.




\end{thebibliography}
\end{document}